\documentclass[11pt]{amsart}
\usepackage{amssymb,mathrsfs,graphicx,enumerate,color}
\usepackage{comment}
\usepackage{hyperref}
\usepackage{colortbl}
\definecolor{black}{rgb}{0.0, 0.0, 0.0}
\definecolor{red}{rgb}{1.0, 0.5, 0.5}

\title[   ]{Decay of large solutions around shocks to multi-D viscous conservation law with strictly convex flux}

\author[Kang]{Moon-Jin Kang}
\author[Oh]{HyeonSeop Oh}
\address[Moon-Jin Kang]{
\newline Department of Mathematical Sciences, \newline Korea Advanced Institute of Mathematical Sciences, Daejeon 34141, Korea}
\email{moonjinkang@kaist.ac.kr}
\address[HyeonSeop Oh]{
\newline Department of Mathematical Sciences, \newline Korea Advanced Institute of Mathematical Sciences, Daejeon 34141, Korea}
\email{ohs2509@kaist.ac.kr}

\newtheorem{theorem}{Theorem}[section]
\newtheorem{lemma}{Lemma}[section]

\newtheorem{remark}{Remark}[section]

\newcommand{\bbr}{\mathbb R}

\newcommand{\e}{\varepsilon}

\numberwithin{figure}{section}
\newcommand{\beq}{\begin{equation}}
\newcommand{\eeq}{\end{equation}}
\newcommand{\bsp}{\begin{split}}
\newcommand{\esp}{\end{split}}

\newcommand{\tiu}{{\tilde{u}}}

\newcommand{\RR}{{\mathbb R}}

\newcommand{\s}{\sigma}

\def\eps{\varepsilon }

\newcommand\adots{\mathinner{\mkern2mu\raise1pt\hbox{.}
\mkern3mu\raise4pt\hbox{.}\mkern1mu\raise7pt\hbox{.}}}

\renewcommand{\div}{{\rm div}}

\newcommand{\intT}{\int_{\TTT^N}}

\newtheorem{theo}{Theorem}[section]
\newtheorem{prop}[theo]{Proposition}

\def\charf {\mbox{{\text 1}\kern-.30em {\text l}}}

\newcommand{\util}{{\tilde{u}}}
\def \l {\lambda}
\def \intR  {\int_\mathbb{R}}
\def \intT  {\int_{\mathbb{T}^2}}
\def \intRT  {\iint_{\mathbb{R} \times \mathbb{T}^2}}
\def \intZT {\int_0^1 \intT} 
\def \intTZ {\intT \int_0^1}

\def \intO {\int_\Omega}
\def \nb {\nabla}
\def \m {\mu}
\def \rd {\partial}
\def \ztil {\widetilde{z}}
\def \d {\delta}
\def \Otil {\widetilde{\Omega}}

\def \intOtil {\int_{\Otil}}
\def \Wbar {\overline{W}}
\def \Mtil {\widetilde{M}}
\def \ubar {\overline{u}}
\def \wtil {\widetilde{w}}
\def \wbar {\overline{w}}

\begin{document}
\bibliographystyle{acm}

\date{\today}


\thanks{\textbf{Acknowledgment.}  This work was supported by Samsung Science and Technology Foundation under Project Number SSTF-BA2102-01.
}

\begin{abstract}
    We consider a planar viscous shock for a scalar viscous conservation law with a strictly convex flux in multi-dimensional setting, where the transversal direction is periodic.
    We first show the contraction property for any solutions evolving from a large bounded initial perturbation in $L^2$ of the viscous shock. The contraction holds up to a dynamical shift, and it is measured by a weighted relative entropy.  This result for the contraction extends the existing result in 1D \cite{Kang19} to the multi-dimensional case.
   As a consequence, if the large bounded initial $L^2$-perturbation is also in $L^1$, then the large perturbation decays of rate $t^{-1/4}$ in $L^2$, up to a dynamical shift that is uniformly bounded in time.
   This is the first result for the quantitative estimate converging to a planar shock under large perturbations. 
   
\end{abstract}

\maketitle \centerline{\date}

\tableofcontents

\section{Introduction}
\setcounter{equation}{0}
We consider the multi-dimensional scalar viscous conservation law with a Lipschitz flux $F:= (f_1, f_2, f_3)$:
\begin{align}
    \begin{aligned} \label{SVCL}
    &u_t + \div_x F(u) = \Delta_x u,\\
    &u(0,x_1,x') = u_0(x_1,x'),
    \end{aligned}
\end{align}
where $u= u(t,x) \in \RR$ is a conserved quantity defined on $t>0, x = (x_1, x')$ with $x_1 \in \RR, x' = (x_2, x_3)\in \mathbb{T}^2 := \bbr^{2} / \mathbb{Z}^{2}$ being $2$-dimensional flat torus. Although our results can be extended to the general $n$-dimensional case: $\mathbb{R} \times \mathbb{T}^{n-1}, n \geq 2$, we only focus on the 3D setting for brevity.  
So we use the simple notation $\Omega := \bbr \times \mathbb{T}^{2}$ in what follows.

For any strictly convex flux $f_1$, and any two constants $u_-, u_+$ with $u_- > u_+$, it is known that \eqref{SVCL} admits a viscous shock waves $ \util(x_1 - \sigma t)$ connecting the two end states $u_-$ and $u_+$. More precisely, there exists a viscous shock $\util$ as a smooth traveling wave solution to the following ODE:
\begin{align}
    \begin{aligned} \label{Vshock}
    &-\sigma \util' + f_1(\util)' = \util'', \quad ' := \frac{d}{d\xi},\\
    &\util(\xi) \to u_\pm, \, \, \text{as } \xi \to \pm \infty,
    \end{aligned}
\end{align}
where the shock speed $\s$ is determined by the Rankine-Hugoniot condition
\begin{equation}\label{RH}
    \sigma = \frac{f_1(u_-)-f_1(u_+)}{u_- - u_+}.
\end{equation}
Moreover, from \eqref{Vshock}-\eqref{RH} and the strict convexity of $f_1$, $\util$ satisfies $u_+ < \util < u_-$ and
\begin{align*}
    \tiu' &= f_1(\tiu) - f_1(u_-) - \sigma(\tiu - u_-)\\
    &= (u_- - \tiu)\left(\frac{f_1(u_-) - f_1(u_+)}{u_- - u_+} - \frac{f_1(u_-) - f_1(\tiu)}{u_- - \tiu} \right) <0.
\end{align*}

There have been many studies on the stability of viscous shock waves in multi-dimensional viscous conservation laws. Goodman \cite{G89} proved the stability of weak shocks using anti-derivative variables with a shift function depending on time and transverse spatial variables. Hoff and Zumbrun \cite{HZ00, HZ02} established the $L^p$-asymptotic stability of large shock waves through the pointwise semigroup method. However, these studies mainly focus on the small perturbations around shock waves. 

In this paper, we aim to show the time-decay estimates of decay rate $t^{-1/4}$ in $L^2(\Omega)$ for large perturbations around a viscous shock of small jump strength, provided that the initial perturbation belongs to $ L^1 \cap L^\infty(\Omega)$. 
These decay estimates can be achieved by the contraction properties of large perturbations in $L^1$ and $L^2$ respectively. 
 
As the existing result \cite {F-S} for the decay of large perturbations to the viscous scalar conservation law, Freist$\ddot{\mbox{u}}$hler and Serre proved the (qualitative) convergence of $L^1$-large perturbations towards a viscous shock in one dimension space. Whereas, our results provide the quantitative decay and contraction estimates for large perturbations in multi-D.
We refer to the $L^1$-contraction result by Kruzkhov \cite{K1}, based on a large family of entropies $\eta_k(u):=|u-k|, ~k\in\bbr$. We also refer to \cite{Bressan,LiuBook} for the fundamentals on $L^1$ theory of shock stability.

However, studying on the decay and contraction of any large perturbations of shocks in $L^2$-distance becomes an important cornerstone for the study on the physical viscous system of conservation laws, since many physical viscous systems of conservation laws (including Navier-Stokes) have only one nontrivial entropy.

\subsection{Main results}
To obtain the main results, we may choose a suitable entropy associated to the given flux, from which  we would control large values of the flux as follows.
First of all, we assume that the flux $f_1$ is strictly convex and satisfies the following exponential growth condition as in \cite{Kang19}:
\begin{equation}\label{growth}
    \text{there exists } a,b>0 \, \, \text{such that} \, \, |f(u)| \leq a e^{b|u|} \,\, \text{for all} \, \, u \in \mathbb{R}.
\end{equation}
Note that any convex function $\eta$ is an entropy of the inviscid scalar conservation law:
\[
u_t + \div_x F(u) = 0,
\]
since there exists an entropy flux $q$ as 
\[
q(u) = \left(\int^u \eta'(v)f_1(v)\,dv, \int^u \eta'(v)f_2(v)\,dv, \int^u \eta'(v)f_3(v)\,dv\right).
\]
However, for our results, we will consider an appropriate entropy $\eta$ associated with the flux $f_1$ that satisfies the following hypotheses:
\begin{enumerate}
    \item[(A1)] There exists a constant $\alpha$ such that $\eta''(u) \geq \alpha$, and $\eta''''(u) \geq \alpha$ for all $u \in \mathbb{R}$.
    \item[(A2)] For a given constant $\theta>0$, there exists a constant $C>0$ such that for any $u,v \in \mathbb{R}$ with $|v| \leq \theta$, the following inequalites hold:\vspace{0.1cm}\\
    (i) $|\eta'(u) - \eta'(v)|^2\mathbf{1}_{\{|u|\leq 2\theta\}} + |\eta'(u) - \eta'(v)|\mathbf{1}_{\{|u|> 2\theta\}} \leq C\eta(u|v)$ \vspace{0.15cm}\\
    (ii) $|f_1(u) - f_1(v)| \leq C|\eta'(u) - \eta'(v)|$ \vspace{0.15cm}\\
    (iii) $|\eta''(u) - \eta''(v)| \leq C|\eta'(u) - \eta'(v)|$ \vspace{0.15cm}\\
    (iv) $|\eta(u) - \eta(v)| \leq C |\eta'(u) - \eta'(v)|$ \vspace{0.15cm}\\
    (v) $\left|\int_v^u \eta''(w)f_1(w)\,dw \right| \leq C(|\eta'(u) - \eta'(v)|\mathbf{1}_{\{|u|\leq 2\theta\}} + |\eta'(u) - \eta'(v)|^2\mathbf{1}_{\{|u|> 2\theta\}})$.
\end{enumerate}
The proof for existence of entropies satisfying the above properties (A1)-(A2) can be found in \cite[Appendix A]{Kang19}. Indeed, the entropy
\[
\eta(u) := e^{bu} + e^{-bu} + u^4 + u^2
\]
satisfies (A1)-(A2), where $b>0$ is defined in \eqref{growth}. Under these assumptions, we consider the relative entropy $\eta(\cdot|\cdot)$ of $\eta$ defined by
\[
\eta(u|v) := \eta(u) - \eta(v) - \eta'(v)(u-v).
\]
Since the entropy functional $u \mapsto \eta(u)$ is strictly convex, the relative entropy is locally quadratic, i.e., 
\[
\text{for any compact set } K \subset \mathbb{R}, \quad \eta(u|v) \sim |u-v|^2, \quad \forall u,v \in K, 
\]
and 
\[
    \eta(u|v) \geq 0, \,\, \text{for any} \, \, u,v, \, \, \text{and} \, \, \eta(u|v) = 0 \iff u=v.
\]
We now state the main results. The first result gives the contraction property of large perturbations of a viscous shock of small strength.
\begin{theorem}\label{main}
    Consider \eqref{SVCL} with a strictly convex flux $f_1$ satisfying the growth condition \eqref{growth}. Let $\eta$ be an entropy which satisfies (A1)-(A2). For any $u_- \in \mathbb{R}$, there exists $\delta_0 \in (0,1)$ such that the following holds.
    
    For any $\e, \l >0$ with $\delta_0^{-1} \e < \l < \delta_0$, let $u_+ := u_- - \e$. Then, there exists a smooth monotone function $a:\mathbb{R} \to \mathbb{R}^+$ with $\lim\limits_{x \to \pm \infty}a(x) = 1 + a_\pm$ for some constants $a_-, a_+$ with $|a_+ - a_-| = \l$ such that the following holds.
    
    Let $\util$ be the viscous shock connecting $u_-$ and $u_+$. Moreover, let $u$ be a solution to \eqref{SVCL} with initial data $u_0$ satisfying $u_0 - \util \in L^2 \cap L^\infty(\Omega)$.  For any $T>0$, there exists an absolutely continuous shift $X(t)$ such that
    \begin{equation}\label{contraction}
        \intO a(x_1 -\s t) \eta(u(t,x_1+X(t),x')|\util(x_1 -\s t))\,dx \leq \intO a(x_1) \eta(u_0(x)|\util(x_1))\,dx, \, \, t \leq T.
    \end{equation}
    Here, the shift function $X(t)$ satisfies $X \in W_{loc}^{1,1}((0,T))$ and
    \begin{equation}\label{shiftestimate}
    |\dot{X}(t)| \leq \frac{1}{\e^2}(1+g(t)), \,\, t \leq T,
    \end{equation}
    for some positive function $g$ satisfying
    \[
    \| g \|_{L^1(0,T)} \leq \frac{2\l}{\delta_0 \e} \intO \eta(u_0|\util)\,dx.
    \]
    \end{theorem}
    \begin{remark}
        By the maximum principle of the scalar viscous conservation law, the relative entropy is comparable with $L^2$ distance. Therefore, \eqref{contraction} implies that
        \begin{equation}\label{L2 type contraction}
            \intO |u(t,x_1+X(t),x') - \util(\cdot - \s t)|^2\,dx \leq C \intO |u_0(x) - \util(x)|^2\,dx
        \end{equation}
        where the constant $C$ depends on $\eta, \delta_0$, and $\|u_0 \|_{L^{\infty}(\Omega)}$.
        \end{remark}
    \begin{remark}
        As in \cite{Kang19}, it suffices to prove Theroem \ref{main} when the shock speed $\s$ is positive(i.e., $\s > 0$). Indeed, if $\s < 0$ or $\s = 0$, then we consider \eqref{SVCL} with a strictly convex flux $g(u) = -2\s u + f(u)$ or $g(u) = \s u + f(u)$ instead of the original flux $f_1$. Therefore, we will assume $\s > 0$ from now on.
        \end{remark}
The following result provides the decay estimate of rate $t^{-1/4}$ in $L^2$  of large perturbations, together with the $L^\infty$-bound of shift function, provided that the initial perturbation belongs to $ L^1 \cap L^\infty(\Omega)$.
\begin{theorem}\label{timedecay} Under the same assumption as in Theorem \ref{main}, suppose that the initial perturbation satisfies $u_0 - \util \in L^1 \cap L^\infty(\Omega)$. Then, there exists a constant $C=C(u_0)>0$ such that 
\begin{equation}\label{time estimate}
    \|u(t,\cdot + X(t), x')  - \util\|_{L^2(\Omega)} \le \frac{C}{1+t^{1/4}}, \quad \forall t>0.
\end{equation}
Moreover, the shift function $X(t)$ satisfies
\begin{equation}\label{L inf shift}
    \|X\|_{L^\infty(\RR_+)} \leq C = C(\|u_0\|_{L^\infty(\mathbb{R})}, \e).
\end{equation}
\end{theorem}

The uniform bound of shift and the time decay rate of the solution can be proved by using Theorem \ref{main} and $L^1$-contraction property of the scalar viscous conservation law \eqref{SVCL}. These proofs are similar to those in \cite{KO24}. Therefore, we will present the proof of Theorem \ref{timedecay} in Appendix \ref{appendix B}.

\begin{remark}
    Even though the Theorem \ref{main} is stated for the domain $\RR \times \mathbb{T}^2$, it can be easily extended to the general $n$-dimensional case: $\mathbb{R} \times \mathbb{T}^{n-1}, n \geq 2$. Moreover, we can show the same time decay rate of $t^{-1/4}$ for the domain $\mathbb{R} \times \mathbb{T}^{n-1}, n \geq 2$.
\end{remark}

\subsection{Main ideas of proof}
To prove the main results, we basically use the method of \(a\)-contraction with shifts, that was developed in \cite{KV16,Vasseur16} for the study of stability of Riemann shock, and extended to the studies for contraction property or long-time dynamics of large perturbations of viscous shock to viscous conservation laws as in  \cite{Kang19,Kang-V-1,KV21,KV-Inven,KVW}.

We here present the new ideas of our proof that distinguish it from the proof in \cite{Kang19} for the 1D case.
To apply the method of $a$-contraction with shifts to large perturbations around shocks, we divide the perturbations into two regions near the shock wave and far from the shock wave.
 For that, we truncate the perturbation $|\eta'(u) - \eta'(\util)|$ with a small constant $\d_1>0$, where the truncation size $\d_1$ is determined by the upper bound of the localized energy as in Lemma \ref{energysmall}. 

To handle the perturbations inside the truncation where $|\eta'(u) - \eta'(\util)|\leq \d_1 $ in the proof of Proposition \ref{inside}, we perform Taylor expansions at $\util$, and apply the multi-dimensional nonlinear Poincar\'e-type inequality of Proposition \ref{multiKV}. In the proof of Proposition \ref{multiKV}, compared to the one-dimensional setting based on the 1D nonlinear Poincar\'e type inequality as in \cite{KV21}, a new difficulty arises in controlling the cubic term $\int_{[0,1] \times \mathbb{T}^2} W^3$ of \eqref{poincare terms}. \\
To control the cubic term of deviation  $\int_{[0,1] \times \mathbb{T}^2} |W - \int_{[0,1] \times \mathbb{T}^2} W|^3$ by the diffusion as in \eqref{w3est}, we need the pointwise estimates for  $|W - \int_{[0,1] \times \mathbb{T}^2} W|$ as in Lemma \ref{KV pointwise} that is sharp along the propagation direction, from which the deviation $|W - \int_{\mathbb{T}^2} W|$ on the transversal domain $\mathbb{T}^2$ would be controlled by the diffusion on $\mathbb{T}^2$ as in \eqref{j2est}. For that, we allow the upper bound of $|W|$ by $1/\eps$ (that is the constraint of Proposition \ref{multiKV} where the scale $\eps$ will be given by the shock strength in \eqref{W bound}), which, via the Poincar\'e inequality on $\mathbb{T}^2$, can be controlled by the diffusion $D_2$ with the coefficient scale $1/\eps^2$ coming from the Jacobian $d\xi/dz$ as in \eqref{d2est}. 

For large values of $u$ such that $|\eta'(u) - \eta'(\util)|\geq \d_1$, we show that the all terms are negligibly smaller than the good terms $G_1, D$ in \eqref{BI}.
In particular, in controlling the worst bad term in \eqref{worstbad}, we need to control the squared average of perturbation over $\mathbb{T}^2$ localized by $a'$, which is obtained after applying the Poincar\'e inequality on $\mathbb{T}^2$. To this end, we use the pointwise estimate \eqref{pwout} with the smallness of energy at a reference point $\xi_0$ due to \eqref{reference point estimate}.

To this end, we obtain the pointwise estimate of $\intT (\eta'(u) - \eta'(\util))\,dx'$ and apply Poincar\'e inequality to control the worst bad term. Finally, we combine Proposition \ref{inside} and Proposition \ref{outside} to complete the proof of Theorem \ref{main}.

To prove Theorem \ref{timedecay}, we keep the diffusion term $D$ from the contraction estimate \eqref{cont conclusion}, and apply Gagliardo Nirenberg type interpolation together with the $L^1$-contraction estimate, to get the $L^2$-decay estimate.

\subsection{Organization of paper}
	The paper is organized as follows. In Section \ref{sec:2}, we present a multi-dimensional nonlinear Poincar\'e type inequality with constraints, that provides a key estimate inside truncation for perturbation. In Section 3, we define shifts and complete the proof of the first theorem, based on the main proposition which will be proved in the following section.

\section{Preliminaries} \label{sec:2}
\setcounter{equation}{0}
In this section, we  provide the properties of the viscous shock waves $\util$ with small amplitude, useful inequalities for the relative entropy, and the time evolution of the relative entropy. Additionally, we present a multi-dimensional nonlinear Poincar\'e type inequality, which extends the one-dimensional case in \cite[Proposition 3.3]{KV21}.
\subsection{Properties of weak shock}
In this subsection, we provide the properties of weak shock waves $\util$, which are useful in stability analysis. 

Without loss of generality, we may assume the viscous shock $\util$ satisfies $\util(0)= \frac{u_- + u_+}{2}$.
\begin{lemma}\label{shock property}\cite[Lemma 2.1]{Kang19}
For any $u_- \in \RR$ with $f_1''(u_-)>0$, there exists $\e_0>0$ such that for any $0<\e<\e_0$, the following is true.

Let $\util$ be the viscous shock wave connecting end states $u_-$ and $u_+:= u_- - \e$ such that $\util(0) =  \frac{u_- + u_+}{2}$. Then, there exists constants $C, C_1, C_2>0$ such that 
\begin{equation}\label{shock der exp decay}
    -C^{-1}\e^2 e^{-C_1 \e |\xi|} \leq \util'(\xi) \leq -C\e^2 e^{-C_2 \e |\xi|}, \, \, \forall \xi \in \mathbb{R}.
\end{equation}
As a consequence, we have
\[
\inf\limits_{[-\frac{1}{\e}, \frac{1}{\e}]} |\util'| \geq C\e^2.
\]
Moreover,
\[
|\util''(\xi)| \leq C\e |\util'(\xi)|, \, \, \forall \xi \in \mathbb{R}.
\]
\end{lemma}
\subsection{Relative entropy method}
To establish the contraction estimate \eqref{contraction}, we use the so-called a-contraction with shift method, which is based on the relative entropy method.

For simplicity, we use the change of variable $(t,x_1,x') \mapsto (t, \xi = x_1 - \s t, x')$ and rewrite \eqref{SVCL} into:
\begin{equation}\label{SVCL xi}
    u_t - \s u_\xi + \div F(u) = \Delta u,
\end{equation}
where the operators $\div$ and $\Delta$ are on $(\xi, x')$ variables.

To utilize the relative entropy method, we first rewrite the viscous term of \eqref{SVCL xi} into
\begin{equation}\label{SVCL G}
    u_t - \div G(u) = \nb \cdot \big(\m(u) \nb(\eta'(u))\big),
\end{equation}
where $G(u)$ and $\m(u)$ are defined by
\[
    G(u) := \big(-\s u + f_1(u), f_2(u), f_3(u) \big) \quad \text{and} \quad \m(u):=\frac{1}{\eta''(u)}.
\]
In addition, \eqref{Vshock} can be written into
\begin{equation}\label{Vshock xi}
    -\s \util_\xi + f_1(\util)_\xi = \nb \cdot \big(\m(\util) \nb(\eta'(\util))\big).
\end{equation}
We will use the notation $\mathcal{F}(\cdot|\cdot)$ to denote the relative functional of $\mathcal{F}$(such as the relative entropy in the case that $\mathcal{F} = \eta$),that is,
\[
\mathcal{F}(u|v) := \mathcal{F}(u) - \mathcal{F}(v) - \mathcal{F}'(v)(u-v).
\]
Let $q(\cdot;\cdot) = \big(q_1(\cdot;\cdot), q_2(\cdot;\cdot),q_3(\cdot;\cdot)\big)$ be the flux of the relative entropy defined by 
\[
q_i(u;v) := q_i(u) - q_i(v) - \eta'(v)(f_i(u) - f_i(v)), \quad i=1,2,3,
\]
where $q = (q_1, q_2, q_3)$ is the entropy flux, i.e., $q_i' = \eta' f_i'$.

For convenience, we introduce the following notations: for any function $f: \mathbb{R}^+ \times \Omega \to  \mathbb{R} $ and the shift $X(t)$, we denote
\[
f^{\pm X} (t, \xi, x') := f(t,\xi \pm X(t), x').
\]
Moreover, we introduce the function space
\[
\mathcal{S} := \{u: u \in L^\infty(\Omega), \quad 
\nb(\eta'(u) - \eta'(\util)) \in L^2(\Omega)\},
\]
on which the below functionals $Y, B, G$ in \eqref{XYBG} are well-defined.

Under the assumption in Proposition \ref{main}, we claim that $u \in \mathcal{S}$. First, it is known that if the initial data $u_0$ satisfies $u_0 \in L^\infty(\Omega)$ and $u_0 - \util \in L^2(\Omega)$, then \eqref{SVCL} admits a bounded weak solution $u$ such that $\|u\|_{L^\infty(\Omega)} \leq \|u_0\|_{L^\infty(\Omega)}$, and 
\begin{equation}\label{energy gronwall}
    \intO |u-\util|^2\,d\xi\,dx' + \int_0^t \intO |\nb (u-\util)|^2\,d\xi\,dx'\,ds \leq e^{Ct} \intO |(u_0 - \util)|^2\,d\xi\,dx', \quad t \leq T.
\end{equation}
Indeed, the estimate \eqref{energy gronwall} can be obtained by the energy method as follows:

Since $u$ and $\util$ are solutions to \eqref{SVCL G}, we have
\begin{align*}
    &\frac{d}{dt} \frac{1}{2} \intO |u -\util|^2\,d\xi\,dx' + \intO |\nb (u-\util)|^2\,d\xi\,dx' = \intO (G(u) - G(\util)) \cdot \nb (u-\util)\,d\xi\,dx'\\
    &\qquad \leq \frac{1}{2}\intO |\nb (u-\util)|^2\,d\xi\,dx' + \frac{1}{2} \intO |G(u) - G(\util)|^2\,d\xi\,dx'\\
    &\qquad \leq \frac{1}{2} \intO |\nb (u-\util)|^2\,d\xi\,dx' +C \intO |u-\util|^2\,d\xi\,dx'.
\end{align*}
Then, the standard Gronwall inequality implies \eqref{energy gronwall}.

Therefore, from $\util' \in L^2(\Omega)$ with the maximum principle $\|u\|_{L^\infty(\Omega)} \leq \|u_0\|_{L^\infty(\Omega)}$, we obtain
\[
\nb(\eta'(u) - \eta'(\util)) = \eta(u) \nb (u-\util) + (\eta''(u) - \eta''(\util))\util' \in L^2((0,T) \times \Omega).
\]
Thus, $u$ satisfies $u \in \mathcal{S}$ for a.e. $t \leq T$.
\begin{lemma}\label{XYBG}
    Let $a:\mathbb{R} \to \mathbb{R}^+$ be a smooth bounded function such that $a'$ is bounded. Let $X:\mathbb{R} \to \mathbb{R}$ be an absolutely continuous function. Let $\util$ be the viscous shock. For any solution $u \in \mathcal{S}$ to \eqref{SVCL xi}, we have
    \[
    \frac{d}{dt} \intO \eta(u^X|\util) \,d\xi\,dx' = \dot{X}Y(u^X) + B(u^X) - G(u^X),
    \]
    where
    \begin{align*}
        Y(u) &= -\intO a \eta(u|\util)\,d\xi\,dx' + \intO a \partial_\xi \eta'(\util)(u-\util) \,d\xi\,dx',\\
        B(u) &= \intO a'\Big[F(u|\util) (\eta'(u) - \eta'(\util))(f_1(u)- f_1(\util)) + f_1(\util) (\eta')(u|\util)\Big] \,d\xi\,dx'\\
        & \quad -\intO a \eta''(\util)\util' f_1(u|\util) \,d\xi\,dx' - \intO a'\mu(u)(\eta'(u) - \eta'(\util))\partial_\xi(\eta'(u) - \eta'(\util)) \,d\xi\,dx'\\
        & \quad - \intO a' (\eta'(u) - \eta(\util))(\mu(u) - \mu(\util)) \eta''(\util)\util' \,d\xi\,dx'\\
        & \quad  - \intO a \partial_\xi(\eta'(u) - \eta(\util))(\mu(u) - \mu(\util)) \eta''(\util)\util' \,d\xi\,dx' + \intO a \util'' (\eta')(u|\util) \,d\xi\,dx',\\
        G(u) &= \sigma \intO a' \eta(u|\util)\,d\xi\,dx' + \intO a \mu(u)|\nabla(\eta'(u) - \eta'(\util))|^2 \,d\xi\,dx',\\
        F(u) &= - \int^u \eta''(v)f_1(v)\,dv.
    \end{align*}
\end{lemma}
\begin{proof}
    Since the proof is essentially the same as in \cite[Lemma 2.2]{Kang19}, we omit the proof.

\end{proof}

\subsection{Construction of the weight function}
We now define the weight function $a$ that satisfies $\s a'>0$ so that the term
\[
\s \intO a' \eta(u^X|\util)\,d\xi\,dx'
\]
is nonnegative. 

First, recall that $\s > 0$. Define the weight function $a=a(\xi)$ by
\begin{equation}\label{def weight}
    a(\xi) := 1 + \l \frac{u_- - \util(\xi)}{\e}.
\end{equation}
Then, the weight function $a$ satisfies 
\[
\lim_{\xi \to \pm \infty} a(\xi) = a_\pm,
\]
and $a' = -\frac{\l}{\e}\util'>0$, so that $\s a' > 0$. Therefore, the functional $G$ in Lemma \ref{XYBG} consists of two good terms.

\subsection{Global estimates on the relative entropy}
Here, we provide several useful inequalities on the relative entropy $\eta(\cdot|\cdot)$.
\begin{lemma}\label{entropy ineq}\cite[Lemma 2.4]{Kang19}
    Let $\eta$ be the entropy satisfying (A1). Then, the followings hold.
    \begin{enumerate}
    \item For any $u, v \in \mathbb{R}$,
    \[
     \frac{\alpha}{2}|u-v|^2 \leq \eta(u|v).
    \]
    \item For any $u_1, u_2, v \in \mathbb{R}$ satisfying $v \leq u_2 \leq u_1$ or $u_1 \leq u_2 \leq v$,
    \[
    \eta(u_1|v) \geq \eta(u_2|v).
    \]
    \end{enumerate}
\end{lemma}

\subsection{Multi-dimensional nonlinear Poincar\'e type inequality with constraints}
In this subsection, we present the multi-dimensional nonlinear Poincar\'e type inequality with certain constraints, which extends the 1-dimensional case in \cite{KV21}. This inequality will be crucially used to prove Proposition \ref{inside}.
\begin{prop}\label{multiKV}
    For a given $M>0$, there exists $\delta_M > 0$ and $\e_M \in (0,1)$ such that for any $0<\delta < \delta_M$ and $0<\e < \e_M$, the following is true.

For any $W(z,x') \in L^2((0,1)\times \mathbb{T}^2)$ with $\sqrt{z(1-z)}\partial_z W \in L^2((0,1) \times \mathbb{T}^2)$, if $W$ satisfies
\[
\|W\|_{L^\infty} \leq \frac{1}{\e}\quad \text{and} \quad \intTZ |W|^2\,dz\,dx' \leq M, 
\]
then we have
\[
R_{\e,\delta}(W) \leq 0,
\]
where
\begin{equation}\label{poincare terms}
\begin{aligned}
R_{\e,\delta}(W) &= -\frac{1}{\delta}\left(\intTZ W^2 \,dz\,dx' + 2\intTZ W \,dz\,dx'\right)^2 + (1+\delta) \intTZ W^2\,dz\,dx'  \\
& + \frac{2}{3}\intTZ W^3\,dz\,dx' + \delta \intTZ |W|^3\,dz\,dx'\\
& - (1-\delta) \intTZ z(1-z)|\rd_z W|^2\,dz\,dx' - \frac{1-\delta}{\e^{3/2}} \sum_{i=2}^3 \intTZ |\rd_{x_i} W|^2\,dz\,dx'.
\end{aligned} 
\end{equation}
\end{prop}
To prove Proposition \ref{multiKV}, we first present a sharp pointwise estimate of $W - \int W$, as a multi-dimensional version of \cite[Lemma 2.8]{KV21}.
\begin{lemma}\label{KV pointwise}
Let $f \in C^1((0,1) \times \mathbb{T}^2)$. Then, for all $(z,y) \in [0,1) \times \mathbb{T}$, we have
\begin{align*}
&f(z,y) - \intTZ f(z',y')\,dz'\,dy'  \\
& \quad \leq \sqrt{L(z) + L(1-z)}\sqrt{\intTZ z(1-z)|\rd_z f(z,y)|^2\,dz\,dx'} + f(z,y) - \intT f(z,y')\,dy',
\end{align*}
where $L(z) := -z-\ln(1-z)$.
\end{lemma}
\begin{proof}
    Note that
    \begin{align*}
    &f(z,x') - \intTZ f(z',y')\,dz'\,dy'\\
    &\quad = \intTZ (f(z,x') - f(z,y'))\,dz'\,dy' + \intTZ (f(z,y') - f(z',y'))\,dz'\,dy' \\
    &\quad= f(z,x') -  \intT f(z,y')\,dy' + \intTZ (f(z,y') - f(z',y'))\,dz'\,dy'.
    \end{align*}
    For the last term, observe that
    \begin{align*}
    &\left|\intTZ (f(z,y') - f(z',y'))\,dz'\,dy' \right|\\
    &\quad \leq \intT \int_0^z \int_{z'}^z |f_z(\ztil,y')|\,d\ztil\,dz'\,dy' + \intT \int_z^1 \int_{z}^{z'} |f_z(\ztil,y')|\,d\ztil\,dz'\,dy' \\
    &\quad \leq \left(\intT \int_0^z \int_{z'}^z \frac{1}{1-\ztil}\,d\ztil\,dz'\,dy'\right)^{1/2} \left(\intT \int_0^z \int_{z'}^z (1-\ztil)|f_z(\ztil,y')|^2\,d\ztil\,dz'\,dy'\right)^{1/2}\\
    &\quad \, +\left(\intT \int_z^1 \int_{z}^{z'} \frac{1}{\ztil}\,d\ztil\,dz'\,dy'\right)^{1/2} \left(\intT \int_z^1 \int_{z}^{z'} \ztil |f_z(\ztil,y')|^2\,d\ztil\,dz'\,dy'\right)^{1/2}\\
    &=: I_1 + I_2.
    \end{align*}
    For $I_1$, we use Fubini theorem: $\int_0^z \int_{z'}^z g \,d\ztil \,dz' = \int_0^z \int_0^{\ztil}  g \,dz' \,d\ztil $ to have
    \begin{align*}
    I_1 &= \left(\intT \int_0^z \int_{z'}^z \frac{1}{1-\ztil}\,d\ztil\,dz'\,dy'\right)^{1/2} \left(\intT \int_0^z \int_{z'}^z (1-\ztil)|f_z(\ztil,y')|^2\,d\ztil\,dz'\,dy'\right)^{1/2}\\
    &=\left(\intT \int_0^z \int_0^{\ztil} \frac{1}{1-\ztil}\,dz'\,d\ztil\,dy'\right)^{1/2} \left(\intT \int_0^z \int_0^{\ztil} (1-\ztil)|f_z(\ztil,y')|^2\,dz'\,d\ztil\,dy'\right)^{1/2}\\
    &= \left(\intT \int_0^z \frac{\ztil}{1-\ztil}\,d\ztil\,dy'\right)^{1/2} \left(\intT \int_0^z \ztil(1-\ztil)|f_z(\ztil,y')|^2\,d\ztil\,dy'\right)^{1/2}\\
    &= (-z - \ln(1-z))^{1/2}\left(\intT \int_0^z \ztil(1-\ztil)|f_z(\ztil,y')|^2\,d\ztil\,dy'\right)^{1/2}.
    \end{align*}
    Likewise, for $I_2$ term, Fubini theorem: $\int_z^1 \int_{z}^{z'} g\,d\ztil \, dz' = \int_z^1 \int_{\ztil}^1 g \,dz'\,d\ztil$ implies that
    \begin{align*}
    I_2 &= \left(\intT \int_z^1 \int_{z}^{z'} \frac{1}{\ztil}\,d\ztil\,dz'\,dy'\right)^{1/2} \left(\intT \int_z^1 \int_{z}^{z'} \ztil |f_z(\ztil,y')|^2\,d\ztil\,dz'\,dy'\right)^{1/2}\\
    &= \left(\intT \int_z^1 \int_{\ztil}^1 \frac{1}{\ztil}\,dz'\,d\ztil\,dy'\right)^{1/2} \left(\intT \int_z^1 \int_{\ztil}^1 \ztil |f_z(\ztil,y')|^2\,dz'\,d\ztil\,dy'\right)^{1/2}\\
    &= (-(1-z) - \ln z)^{1/2} \left(\intT \int_z^1  \ztil(1-\ztil) |f_z(\ztil,y')|^2 \,d\ztil\,dy'\right)^{1/2}.
    \end{align*}
    Now, let $L(z):= -z-\ln(1-z)$, and 
    \begin{align*}
    X := \intT \int_0^z \ztil(1-\ztil)|f_z(\ztil,y')|^2\,d\ztil\,dy', \quad   D := \intT \int_0^1 \ztil(1-\ztil)|f_z(\ztil,y')|^2\,d\ztil\,dy'.
    \end{align*}
    Using these notations, we have
    \[
    I_1 + I_2 = \sqrt{L(z)}\sqrt{X} + \sqrt{L(1-z)}\sqrt{D-X}.
    \]
    Using the Cauchy-Schwartz inequality $(ab+cd)^2 \leq (a^2+c^2)(b^2+d^2)$, we obtain
    \begin{align*}
    I_1 + I_2 \leq \sqrt{L(z) + L(1-z)} \sqrt{D}.
    \end{align*}
    \end{proof}
Through an elementary calculation, we find that(see \cite{KV21})
\begin{equation}\label{theta def}
    \left(\int_0^1 (L(z) + L(1-z))^2 \,dz \right)^{1/2} = \sqrt{5 - \frac{\pi^2}{3}}=: \theta.
\end{equation}
The next result is the algebraic inequality on a specific polynomial. we consider the following polynomials:
\begin{align*}
&E(Z_1, Z_2):= Z_1^2 + Z_2 ^2 + 2Z_1,\\
&P_\delta(Z_1, Z_2) := (1+\delta)(Z_1^2 + Z_2 ^2) + 2 Z_1 Z_2^2 + \frac{2}{3}Z_1^3 + 6\delta(|Z_1|Z_2^2 + |Z_1|^3)\\
&\phantom{P_\delta(Z_1, Z_2) := } -2\left(1-\delta - \left(\frac{2}{3} + \delta \right) \frac{21}{20} \theta Z_2 \right)Z_2^2,
\end{align*}
where $\theta$ is defined in \eqref{theta def}.
\begin{lemma}\label{algebraic}
    There exists $\delta_2, \delta_3>0$ such that for any $0 < \delta < \delta_2$, the following is true.
    
    If $(Z_1, Z_2) \in \mathbb{R}^2$ satisfies $|E(Z_1, Z_2)| \leq \delta_3$, then
    \begin{equation}\label{algebraic inequality}
        P_\delta(Z_1, Z_2) - |E(Z_1, Z_2)|^2 \leq 0.
    \end{equation}
\end{lemma}
The proof of Lemma \ref{algebraic} is almost same as in \cite[Proposition 3.2]{KV21}. We will give a proof of Lemma \ref{algebraic} in Appendix \ref{proof alg} for reader's convenience.

Now, we present the Poincar\'e type inequality on $[0,1] \times \mathbb{T}^2$, which is the extension of the 1-dimensional Poincar\'e type inequality in \cite{KV21}.
\begin{lemma}\label{Poincare}\cite{Wang22}
    (Poincar\'e type inequality without constraints) \\
    For $W:[0,1] \times \mathbb{T}^2 \to \mathbb{R}$ such that
    \[
    \intTZ z(1-z)|\rd_zW|^2 \,dz\,dx' + \sum_{i=2}^3  \intTZ |\rd_{x_i} W|^2 \,dz\,dx' < +\infty,
    \]
    we have
    \[
    \intTZ |W - \Wbar|^2 \,dz\,dx' \leq \frac{1}{2} \intTZ z(1-z)|\rd_zW|^2 \,dz\,dx' + \frac{1}{4\pi^2}\sum_{i=2}^3 \intTZ |\rd_{x_i} W|^2 \,dz\,dx',
    \]
    where $\displaystyle{\Wbar :=  \intTZ W \,dz\,dx'}$, the average of $W$.
\end{lemma}
$\bullet$ \textit{Proof of Proposition \ref{multiKV}:}
For convenience, we use the notation $\widetilde{\Omega} := [0,1] \times \mathbb{T}^2$.
We first separate the cubic term in $R_{\e,\delta}$ into three parts:
\begin{align*}
\intOtil W^3 \,dz\,dx' &= \intOtil ((W-\Wbar) + \Wbar)^3 \,dz\,dx'\\
&= \intOtil (W-\Wbar)^3 \,dz\,dx' + 2\Wbar \intOtil (W-\Wbar)^2 \,dz\,dx' + \Wbar  \intOtil W^2 \,dz\,dx'.
\end{align*}
Then, $R_{\e,\delta}$ can be expressed as follows:
\begin{align*}
R_{\e,\delta}(W) &= -\frac{1}{\delta}\left(\intOtil W^2 \,dz\,dx' + 2\intOtil W \,dz\,dx'\right)^2 + (1+\delta) \intOtil W^2\,dz\,dx'  \\
& + \frac{4}{3}\Wbar \intOtil (W-\Wbar)^2 \,dz\,dx' + \frac{2}{3}\intOtil (W-\Wbar)^3 \,dz\,dx'+ \frac{2}{3} \Wbar  \intOtil W^2 \,dz\,dx' \\
& + \delta \intOtil |W|^3\,dz\,dx' - (1-\delta) \intOtil z(1-z)|\rd_z W|^2\,dz\,dx' - \frac{1-\delta}{\e^{3/2}} \sum_{i=2}^3 \intOtil |\rd_{x_i} W|^2\,dz\,dx'.
\end{align*}
Now, let
\[
Z_1 := \Wbar, \quad Z_2 := \left(\intOtil (W- \Wbar)^2 \,dz\,dx'\right)^{1/2}, \quad E(Z_1, Z_2):= Z_1^2 + Z_2^2 + 2Z_1.
\]
Note that
\[
\intOtil W^2 \,dz\,dx' = Z_1^2 + Z_2^2,
\]
and
\begin{align*}
\intOtil |W|^3 \,dz\,dx' &\leq \intOtil |W- \Wbar|^3 \,dz\,dx' + 3|\Wbar| \intOtil |W- \Wbar|^2 \,dz\,dx'\\
&\quad + 3|\Wbar|^2 \intOtil |W- \Wbar| \,dz\,dx' + |\Wbar|^3\\
&\leq \intOtil |W- \Wbar|^3 \,dz\,dx' + 3|Z_1|Z_2^2 + 3|Z_1|^{3/2}(|Z_1|^{1/2}Z_2) + |Z_1|^3\\
&\leq \intOtil |W- \Wbar|^3 \,dz\,dx' + 6|Z_1|Z_2^2 + 4|Z_1|^3.
\end{align*}
Thus, we have
\begin{align*}
R_{\e,\delta}(W) &\leq -\frac{1}{\delta}|E(Z_1,Z_2)|^2 + (1+\delta)(Z_1^2 + Z_2^2) + \frac{4}{3}Z_1Z_2^2 + \frac{2}{3}Z_1(Z_1^2+ Z_2^2)\\
&\quad +\left(\frac{2}{3}+\delta \right) \intOtil |W- \Wbar|^3 \,dz\,dx' + 6\delta(|Z_1|Z_2^2 + |Z_1|^3) \\
&\quad - (1-\delta) \intOtil z(1-z)|\rd_z W|^2\,dz\,dx' - \frac{1-\delta}{\e^2} \intOtil |\rd_y W|^2\,dz\,dx'\\
&=:  -\frac{1}{\delta}|E(Z_1,Z_2)|^2 + (1+\delta)(Z_1^2 + Z_2^2)  + 2Z_1Z_2^2 + \frac{2}{3}Z_1^3 + 6\delta (|Z_1|Z_2^2 + |Z_1|^3) + \mathcal{P},
\end{align*}
where
\[
\mathcal{P} := \left(\frac{2}{3}+\delta \right) \intOtil |W- \Wbar|^3 \,dz\,dx' - (1-\delta) \intOtil z(1-z)|\rd_z W|^2\,dz\,dx' - \frac{1-\delta}{\e^{3/2}} \intOtil |\rd_y W|^2\,dz\,dx'.
\]
For the cubic term in $\mathcal{P}$, we use Lemma \ref{KV pointwise} to have
\begin{align*}
&\intOtil |W- \Wbar|^3 \,dz\,dx'\\
& \leq \intOtil \left[\sqrt{L(z) - L(1-z)} \sqrt{\intOtil z(1-z)|W_z|^2\,dz\,dx'} \right.\\
&\phantom{\leq \intOtil \sqrt{L(z)} } \left. + \left(W(z,y) - \intT W(z,y')\,dy'\right) \right]^2 |W-\Wbar|(z,y)\,dz\,dx'\\
& \leq \frac{21}{20}\left(\intOtil z(1-z)|W_z|^2\,dz\,dx'\right) \intOtil (L(z) - L(1-z))|W-\Wbar|\,dz\,dx'\\
&\quad + 21 \intOtil \left(W(z,y) - \intT W(z,y')\,dy'\right)^2 |W-\Wbar| \,dz\,dx'\\
& := I_1 + I_2.
\end{align*}
For $I_1$, we use H\"older inequality to have
\begin{align*}
I_1 &= \frac{21}{20} \left(\intOtil z(1-z)|W_z|^2\,dz\,dx' \right) \intOtil (L(z) - L(1-z))|W-\Wbar|\,dz\,dx'\\
&\leq \frac{21}{20}\theta \left(\intOtil z(1-z)|W_z|^2\,dz\,dx'\right)\left(\intOtil |W-\Wbar|^2\,dz\,dx'\right)^{1/2}\\
&=  \frac{21}{20}\theta Z_2 \intOtil z(1-z)|W_z|^2\,dz\,dx'.
\end{align*}
By the Poincar\'e inequality and the assumption $\|W - \Wbar\|_{L^\infty(\Omega)} \leq \frac{2}{\e}$, $I_2$ term can be estimated as follows:
\begin{align*}
I_2 &=  21 \intOtil \left(W(z,y) - \intT W(z,y')\,dy'\right)^2 |W-\Wbar| \,dz\,dx'\\
&\leq \frac{42}{\e} \intZT \left(W(z,y) - \intT W(\def \intZT {\int_0^1 \intT} z,y')\,dy'\right)^2 \,dx'\,dz\\
&\leq \frac{C}{\e} \sum_{i=2}^3 \intOtil |W_{x_i}|^2 \,dz\,dx'.
\end{align*}
Thus, we obtain
\begin{align*}
\mathcal{P} &\leq -\left(1-\delta + \left(\frac{2}{3} + \delta\right) \frac{21}{20}\theta Z_2 \right) \intOtil z(1-z)|W_z|^2\,dz\,dx'\\
&\quad +\frac{C}{\e}\sum_{i=2}^3 \intOtil |W_{x_i}|^2 \,dz\,dx' - \frac{1-\delta}{\e^{3/2}} \sum_{i=2}^3 \intOtil |W_{x_i}|^2 \,dz\,dx' .
\end{align*}
Since $(Z_1+1)^2 + Z_2^2 = 1+E(Z_1, Z_2)$, observe that
\[
Z_2 \leq \sqrt{1+ |E(Z_1, Z_2)|}.
\]
Since $\frac{2}{3}\cdot \frac{21}{20}\theta = 0.7\theta \approx 0.915 <1$, there exists a positive constant $\delta_{\theta} <1$ such that
\[
\frac{2}{3}\cdot \frac{21}{20}\theta \sqrt{1+\delta_\theta} < 1.
\]
Then, we take $\delta_M<1$ satisfying 
\[
1- \delta - \left(\frac{2}{3}+ \delta\right)\frac{21}{20} \theta \sqrt{1+\delta_\theta} >0, \quad \forall \delta < \delta_M.
\]

Now, we divide into two cases: $|E(Z_1, Z_2)| \leq \min\{\delta_\theta, \delta_3\}$ or $|E(Z_1, Z_2)| \geq \min\{\delta_\theta, \delta_3\}$, where $\delta_3$ is the constant in Lemma \ref{algebraic}.\vspace{0.3cm}\\
\textit{Case I:}  $|E(Z_1, Z_2)| \leq \min\{\delta_\theta, \delta_3\}$.\\
Note that
\[
Z_2 \leq \sqrt{1 + |E(Z_1, Z_2)|} \leq \sqrt{1+\min\{\delta_\theta, \delta_3\}}.
\]
Then, we find that for any $\delta < \delta_M$,
\begin{align*}
1- \delta - \left(\frac{2}{3}+ \delta\right)\frac{21}{20} \theta Z_2 &\geq 1- \delta - \left(\frac{2}{3}+ \delta\right)\frac{21}{20} \theta \sqrt{1+\min\{\delta_\theta,\delta_3\}}\\
& \geq1- \delta - \left(\frac{2}{3}+ \delta\right)\frac{21}{20} \theta \sqrt{1+\delta_\theta}\\
& > 0.
\end{align*}
Now, we use Lemma \ref{Poincare}:
\[
Z_2^2 \leq \frac{1}{2}\intOtil z(1-z)|W_z|^2\,dz\,dx' + \frac{1}{4\pi^2}\sum_{i=2}^3 \intOtil |\rd_{x_i} W|^2\,dz\,dx'
\]
to have
\begin{align*}
\mathcal{P} &\leq -2\left(1-\delta + \left(\frac{2}{3} + \delta\right) \frac{21}{20}\theta Z_2 \right)Z_2^2\\
&\quad  + \left[ \frac{1}{2\pi^2}\left(1-\delta + \left(\frac{2}{3} + \delta\right) \frac{21}{20}\theta Z_2 \right) + \frac{C}{\e} - \frac{1-\delta}{\e^{3/2}}  \right] \sum_{i=2}^3 \intOtil |\rd_{x_i} W|^2\,dz\,dx'\\
&\leq -2\left(1-\delta + \left(\frac{2}{3} + \delta\right) \frac{21}{20}\theta Z_2 \right)Z_2^2 + \left(\frac{1}{2\pi^2} + \frac{C}{\e} - \frac{(1-\delta)}{\e^{3/2}}\right)\sum_{i=2}^3 \intOtil |\rd_{x_i} W|^2\,dz\,dx'.
\end{align*}
If we take $\e>0$ small enough, we obtain
\[
    \mathcal{P} \leq -2\left(1-\delta + \left(\frac{2}{3} + \delta\right) \frac{21}{20}\theta Z_2 \right).
\]
Thus, we have
\begin{align*}
R_{\e,\delta}(W) &\leq  -\frac{1}{\delta}|E(Z_1,Z_2)|^2 + (1+\delta)(Z_1^2 + Z_2^2) + 2Z_1Z_2^2 + \frac{2}{3}Z_1^3 + 6\delta (|Z_1|Z_2^2 + |Z_1|^3)\\
&\quad + 2\left(1-\delta + \left(\frac{2}{3} + \delta\right) \frac{21}{20}\theta Z_2 \right)Z_2^2 \\
& = -\frac{1}{\delta} |E(Z_1, Z_2)|^2 + P_\delta (Z_1, Z_2).
\end{align*}
Hence, if we retake $\delta_M$ such that $\delta_M < \min\{\delta_2, 1\}$, Lemma \ref{algebraic} implies that $R_{\e,\delta}(W) \leq 0$ for all $\delta < \delta_M$, whenever $|E(Z_1, Z_2)| \leq \min\{\delta_\theta, \delta_3\}$.\vspace{0.3cm}\\
\textit{Case II:} $|E(Z_1, Z_2)| \geq \min\{\delta_\theta, \delta_3\}$.

We now use the assumption $\displaystyle{\intOtil W^2 \,dz\,dx' \leq M}$. Since
\[
|Z_1|^2 \leq  \intOtil W^2 \,dz\,dx' \leq M
\] 
and 
\[
    Z_2 =   \intOtil W^2 \,dz\,dx' -   \left(\intOtil W \,dz\,dx'\right)^2 \leq 2  \intOtil W^2 \,dz\,dx' \leq 2M,
\]
All bad terms except $\intOtil |W- \Wbar|^3 \,dz\,dx'$ are bounded by some constant $\Mtil = \Mtil(M)$. Therefore, we have
\begin{align*}
R_{\e,\delta}(W) &\leq -\frac{1}{\delta} \min\{\delta_\theta, \delta_3 \}^2 + \Mtil + \frac{2}{3}\left(1+\frac{3}{2}\delta\right)\intOtil |W- \Wbar|^3 \,dz\,dx'\\
&\quad - (1-\delta) \intOtil z(1-z)|\rd_z W|^2\,dz\,dx' - \frac{1-\delta}{\e^{3/2}} \sum_{i=2}^3 \intOtil |\rd_{x_i} W|^2\,dz\,dx'.
\end{align*}
For the remaining cubic term, we use Lemma \ref{KV pointwise} to have
\beq\label{w3est}
\begin{aligned}
&\intOtil |W- \Wbar|^3 \,dz\,dx'\\
&\leq \intOtil \left[\sqrt{L(z) - L(1-z)} \sqrt{\intOtil z(1-z)|W_z|^2\,dz\,dx'} \right.\\
&\phantom{\leq \intOtil \sqrt{L(z) }} \left. + \left(W(z,x') - \intT W(z,y')\,dy'\right) \right]^{3/2} |W-\Wbar|^{3/2}\,dz\,dx'\\
&\leq 2^{\frac{3}{2}} \left[ \left(\intOtil z(1-z)|W_z|^2\,dz\,dx'\right)^{3/4} \intOtil \big(L(z) - L(1-z)\big)^{3/4} |W-\Wbar|^{3/2}\,dz\,dx' \right.\\
&\quad \quad \left. + \intOtil  \left(W(z,x') - \intT W(z,y')\,dy'\right)^{3/2}|W-\Wbar|^{3/2}\,dz\,dx' \right]\\
&:= J_1 + J_2.
\end{aligned}
\eeq
For $J_1$, we get
\begin{align*}
J_1 &= 2^{\frac{3}{2}} \left(\intOtil z(1-z)|W_z|^2\,dz\,dx'\right)^{3/4} \intOtil (L(z) - L(1-z))^{3/4} |W-\Wbar|^{3/2}\,dz\,dx' \\
&\leq  2^{\frac{3}{2}} \left(\intOtil z(1-z)|W_z|^2\,dz\,dx'\right)^{3/4} \left(\intOtil (L(z) - L(1-z))^{3}\,dz\,dx'\right)^{1/4} \left(\intOtil |W-\Wbar|^{2}\right)^{3/4}\\
& \leq \Mtil  \left(\intOtil z(1-z)|W_z|^2\,dz\,dx'\right)^{3/4}\\
&\leq \Mtil + \frac{1}{10}\left(\intOtil z(1-z)|W_z|^2\,dz\,dx'\right)^2.
\end{align*}
Thanks to the Poincar\'e inequality and the assumption $\|W-\Wbar\|_{L^\infty(\Otil)} \leq \frac{C}{\e}$, $J_2$ can be estimated as follows:
\beq\label{j2est}
\begin{aligned}
J_2 &= 2^{\frac{3}{2}}\intOtil  \left(W(z,x') - \intT W(z,y')\,dy'\right)^{3/2}|W-\Wbar|^{3/2}\,dz\,dx'\\
&\leq 2^{\frac{3}{2}}\left(\intOtil  \left(W(z,x') - \intT W(z,y')\,dy'\right)^{2}\,dz\,dx'\right)^{3/4}\left(\intOtil |W-\Wbar|^{6}\,dz\,dx' \right)^{1/4}\\
&\leq \frac{C}{\e}  \left(\intZT \left(W(z,x') - \intT W(z,y')\,dy'\right)^{2} \,dx'\,dz\right)^{3/4} \left(\intOtil |W-\Wbar|^{2}\,dz\,dx' \right)^{1/4}\\
&\leq \frac{C}{\e} \Mtil \left(\sum_{i=2}^3 \intOtil |\rd_{x_i} W|^2\,dz\,dx'\right)^{3/4}\\
&\leq \frac{C}{\e^{4/3}}\sum_{i=2}^3 \intOtil |\rd_{x_i} W|^2\,dz\,dx' + \Mtil.
\end{aligned}
\eeq
Therefore, taking $\e>0$ and $\d>0$ small enough, we get
\begin{align*}
R_{\e,\delta}(W) &\leq -\frac{1}{\delta}\min\{\delta_{\theta}, \delta_3\} + 3\Mtil + \left[\frac{2}{30} \left(1+\frac{3}{2}\delta \right)- (1-\delta) \right] \intOtil z(1-z)|\rd_z W|^2\,dz\,dx'\\
& \quad + \left[\frac{C}{\e^{4/3}} - \frac{(1-\delta)}{\e^{3/2}}  \right] \sum_{i=2}^3 \intOtil |\rd_{x_i} W|^2\,dz\,dx' \\
&\leq -\frac{1}{\delta}\min\{\delta_{\theta}, \delta_3\} + 3\Mtil.
\end{align*}
Hence, choosing $\delta_M < \min\{\delta_2,1\}$ satisfying $-\frac{1}{\delta_M} \min\{\delta_{\theta}, \delta_3\} + 3\Mtil < 0$, we get $R_{\e,\delta}(W) < 0$ whenever $|E(Z_1, Z_2)| \geq \min\{\delta_\theta, \delta_3\}$.
\qed
\section{Proof of Theorem \ref{main}}
\setcounter{equation}{0}
\subsection{Definition of the shift function}
For any fixed $\e> 0$, we consider a continuous function $\Phi_\e$ defined by
\begin{equation}\label{truncationshift}
\Phi_\e (y):= \begin{cases}
\frac{1}{\e^2} \,\, &\text{if } \, \, y \leq -\e^2,\\
-\frac{1}{\e^4} y \,\, &\text{if } \, \, |y| \leq \e^2,\\
-\frac{1}{\e^2} \,\, &\text{if } \, \, y \geq \e^2.
\end{cases}
\end{equation}
Define a shift function $X(t)$ as a solution of following ODE:
\begin{equation}\label{shiftdef}
\begin{cases}
&\dot{X}(t) = \Phi_\e (Y(u^X))(2|B(u^X)|+1),\\
&X(0) = 0.
\end{cases}
\end{equation}
Indeed, for any $T>0$, the solution of the above ODE exists uniquely, and is absolutely continuous on $[0,T]$. The proof of the existence of the shift function can be found in \cite{Kang19}.
\subsection{Proof of Theorem \ref{main} from a main proposition}
\begin{prop}\label{mainprop}
    There exists $\delta_0 \in (0,1)$ such that for any $\e, \l$ with $\delta_0^{-1} \e < \l < \delta_0 < 1/2$, the following is true.
    
    For any $u \in \{u| |Y(u) \leq \e^2 \}$,
    \[
    R(u) := -\frac{1}{\e^4}|Y(u)|^2 + B(u) + \delta_0 \frac{\e}{\l}|B(u)| - G_0(u) - (1-\d_0)D(u) \leq 0,
    \]
    where $Y$ and $B$ are as in \eqref{XYBG}, and $G_0$ and $D(u)$ are defined by
    \begin{align*}
        G_0(u) := \sigma  \intO a'\eta(u|\util) \,d\xi\,dx', \quad D(u) := \intO a \m(u)|\nb(\eta'(u) - \eta'(\util))|^2\,d\xi\,dx'.
    \end{align*}
    \end{prop}
    Now, we will show that Proposition \ref{mainprop} indeed implies Theorem \ref{main}.
    By the definition of shift \eqref{shiftdef}, it suffices to prove that for almost every time $t>0$,
    \[
    \Phi_\e (Y(u^X))(2|B(u^X)|+1)Y(u^X) + B(u^X) - G(u^X) \leq 0.
    \]
    For every $u$, we define
    \[
    \mathcal{F}(u) := \Phi_\e (Y(u))(2|B(u)|+1)Y(u) + B(u) - G(u).
    \]
    From \eqref{truncationshift}, we have
    \[
    \Phi_\e (Y)(2|B|+1)Y \leq \begin{cases}
    -2|B| \, \, &\text{if }  |Y| \geq \e^2,\\
    -\frac{1}{\e^4} Y^2, \, \, &\text{if } |Y| \leq \e^2.
    \end{cases}
    \]
    Hence, for all $u$ satisfying $|Y(u)| \geq \e^2$, we have
    \[
    \mathcal{F}(u) \leq - |B(u)| - G(u) \leq 0.
    \]
    By Proposition \ref{mainprop}, we find that for all $u$ satisfying $|Y(u)| \leq \e^2$,
    \[
    \mathcal{F}(u) \leq - \delta_0 \left(\frac{\e}{\l}\right)|B(u)|- \d_0 D(u) \leq 0.
    \]
    Since $\delta_0 < 1$ and $\frac{\e}{\l}< \delta_0$, we have
    \[
    \mathcal{F}(u) \leq - \delta_0 \left(\frac{\e}{\l}\right)|B(u)|  - \d_0 D(u)\, \, \text{for any } u.
    \]
    Thus, for every fixed $t>0$, using this estimate with $u=u^X(t,\cdot,\cdot)$, we obtain
    \begin{equation}\label{cont conclusion}
    \begin{aligned}
        &\frac{d}{dt} \intO a \eta(u^X|\util)\,d\xi\,dx' \leq \mathcal{F}(u^X) \leq  - \delta_0 \left(\frac{\e}{\l}\right)|B(u^X)|- \d_0 D(u^X) \leq 0,
    \end{aligned}
    \end{equation}
    which completes the contraction estimate \eqref{contraction}.
    Moreover, \eqref{cont conclusion} implies that
    \[
    \delta_0 \left(\frac{\e}{\l}\right)|B(u^X)| \leq \intO \eta(u_0|\util)\,d\xi\,dx' < \infty,
    \]
    using \eqref{shiftdef} and $\|\Phi_\e \|_{L^\infty(\Omega)} \leq \frac{1}{\e^2}$, we get
    \[
    |\dot{X}| \leq \frac{1}{\e^2} + \frac{2}{\e^2}|B|, \quad \|B \|_{L^1(\Omega)} \leq \frac{1}{\delta_0} \frac{\l}{\e} \intO \eta(u_0|\util)\,d\xi\,dx',
    \]
    which provides the global-in-time estimate for the shift function \eqref{shiftestimate}, and thus $X \in W^{1,1}_{loc}((0,T))$.
    \qed
\section{Proof of Proposition \ref{mainprop}}
\setcounter{equation}{0}
\subsection{Expansion in the size of the shock}


First, we truncate the nonlinear perturbation $|\eta'(u) - \eta'(\util)|$ with small parameter $\d_1>0$. For values of $u$ such that $|\eta'(u) - \eta'(\util)| \leq \d_1 \ll 1$, we use Taylor expansion at $\util$ for the terms $Y, G$, and $B_I$ in \eqref{BI}.

Now, we recall the functionals $Y$ and $F$, $G_0$ and define the following new functionals:
\begin{equation}\label{BI}
    \begin{aligned}
    &Y(u):= - \intO a' \eta(u|\util) \,d\xi\,dx'+ \intO a \partial_\xi \eta'(\util)(u-\util) \,d\xi\,dx',\\
    &F(u):= -\int^u \eta''(v)f(v)\,dv,\\
    &B_I(u):=  \intO a' \Big[F(u|\util) + (\eta'(u) - \eta'(\util))(f(u) - f(\util)) + f(\util) (\eta')(u|\util)\Big] \,d\xi\,dx', \\
    &\phantom{B_I(u):=} - \intO a \eta''(\util)\util' f(u|\util) \,d\xi\,dx'\\ 
    &G_0(u) := \sigma  \intO a'\eta(u|\util) \,d\xi\,dx', \\
    &D_1(u) := \intO a \m (u) |\partial_\xi(\eta'(u) - \eta'(\util))|^2 \,d\xi\,dx', \\
    &D_2(u) : =\sum_{i=2}^3 \intO a \m (u) |\partial_{x_i} \big(\eta'(u) - \eta'(\util)\big)|^2 \,d\xi\,dx', \\
    &D(u) := D_1(u) + D_2(u).
    \end{aligned}
    \end{equation}
\begin{prop}\label{inside}
    Let $\eta$ be the entropy satisfying the entropy assumption (A1) - (A2). For any $K>0$, there exists $\delta_1 \in (0,1)$ such that for any $\delta_1^{-1}\e < \l < \delta_1$ and for any $\delta \in (0,\delta_1)$, the following is true.
    
    For any function $u:\mathbb{R} \to \mathbb{R}$ such that $D(u) + G_0(u)$ is finite, if
    \begin{equation}\label{K assumption}
        |Y(u)| \leq K \frac{\e^2}{\l} \, \, \text{and} \, \, \|\eta'(u) - \eta(\util)\|_{L^\infty(\Omega)} \leq \delta_1,
    \end{equation}
    then we have
    \begin{align*}
    \mathcal{S}_{\e, \delta}(u)&:= -\frac{1}{\e \delta}|Y(u)|^2 + B_I + \delta \frac{\e}{\l}\left|B_I\right|- \left(1-\delta \frac{\e}{\l}\right)G_0(u) - (1-\delta)D(u) \leq 0.
    \end{align*}
    \end{prop}
    \begin{proof}
    From the entropy hypothesis (A1), we obtain
    \[
    \|u- \util\|_{L^\infty(\Omega)} \leq \alpha^{-1}\d_1.
    \]

    Now, we introduce a change of varible $\xi \in \mathbb{R} \mapsto z \in [0,1]$ defined by
\[
z(\xi) = \frac{u_- - \util(\xi)}{\e}.
\]
Since $\util'<0$, this change of variable is well-defined. Note that
\[
\util' \,d\xi = -\e \,dz \, \, \text{and } \, \, a' \,d\xi = \l \,dz.
\]
Moreover, we denote
\[
w(t,z,x') := (u-\util) \circ (z^{-1},id),
\]
and normalized function
\[
W(t,z,x') := \frac{\l}{\e}w(t, z,x'),
\]
where $id : \mathbb{T}^{2} \to  \mathbb{T}^{2}$ is the identity operator on $ \mathbb{T}^{2}$.
Observe that if we choose $\delta_1$ small enough such that $\delta_1 \leq \alpha$, then if 
\[
\|\eta'(u) - \eta'(\util)\|_{L^\infty (\mathbb{R} \times \mathbb{T}^2)} \leq \delta_1,
\]
we obtain
\begin{equation}\label{W bound}
\begin{aligned}
\|W \|_{L^\infty ([0,1] \times \mathbb{T}^2)} = \frac{\l}{\e}\|u - \util \|_{L^\infty (\Omega)} \leq \frac{\l}{\e}\frac{1}{\alpha}\|\eta'(u) - \eta'(\util)\|_{L^\infty (\Omega)} \leq \frac{\l}{\e} \frac{\delta_1}{\alpha} \leq \frac{1}{\e}.
\end{aligned} 
\end{equation}
In what follows, we use Taylor theorem at $\util$ for the functionals $Y, B_1, \cdots, B_4$, and $G$. After expanding these functionals, we apply Proposition \ref{multiKV} to derive Proposition \ref{inside}.

First of all, as in the proof of \cite[Proposition 4.1]{Kang19}, we have the followings:\vspace{0.15cm}

1. There exists $M = M(K)>0$ such that
    \begin{equation}\label{M constant}
        \intOtil W^2\,dz\,dx' \leq M,
    \end{equation}
where $\Otil := [0,1] \times \mathbb{T}^2$. Here, the constant $M$ does not depend on the size of the truncation $\d_1$. \vspace{0.15cm}

2. Moreover, we have
    \begin{align*}
        &\frac{2}{\eta''(u_-) f''(u_-)} \frac{\l^2}{\e^3} \left[-\frac{1}{\e \delta}|Y(u)|^2 + B_I + \delta \frac{\e}{\l}\left|B_I\right|- \left(1-\delta \frac{\e}{\l}\right)G_0(u) - (1-\delta)D_1(u) \right]\\
        &\leq -\frac{1}{C\delta_1} \left(\intOtil W^2\,dz\,dx' + 2\intOtil W\,dz\,dx' \right)^2 + (1+C\delta_1) \intOtil W^2\,dz\,dx' + \frac{2}{3}\intOtil W^3\,dz\,dx' \\
        &\quad + C\delta_1 \intOtil |W|^3 dz dx' - (1-C\delta_1)\intOtil z(1-z)|\rd_z W|^2 dz dx'.
        \end{align*}
Thus, it remains to estimate $D_2$.\vspace{0.5cm}\\
$\bullet$ (Estimate of $D_2$) In order to estimate the diffusion term, we need to approximate the Jacobian $\frac{dz}{d\xi}$. This is provided by the following lemma.
\begin{lemma}\label{diffusion}\cite[Lemma 4.3]{Kang19}
    There exists constant $C>0$ such that for any $\e < \delta_1$, and any $z \in [0,1]$,
    \[
    \left| \frac{1}{z(1-z)}\frac{dz}{d\xi} - \e \frac{f_1''(u_-)}{2}\right| \leq C \e^2.
    \]
\end{lemma}
Now, we estimate $D_2$. Since
\begin{align*}
D_2 & = \sum_{i=2}^3 \intO a \frac{1}{\eta''(u)} \left|\eta''(u)\rd_{x_i} (u-\util) + (\eta''(u) - \eta''(\util))\util_{x_i} \right|^2\,d\xi\,dx'\\
&= \sum_{i=2}^3 \intO a \frac{1}{\eta''(u)} \left|\eta''(u)\rd_{x_i} (u-\util)\right|^2\,d\xi\,dx',
\end{align*}
we have
\[
-  D_2 \leq -(1-C\delta_1) \eta''(u_-) \sum_{i=2}^3 \intOtil |w_{x_i}|^2 \left(\frac{d\xi}{dz}\right) \,dz\,dx'.
\]
Therefore, using Lemma \ref{diffusion} and the algebraic inequality: $1/(z(1-z)) \geq 4$ on $[0,1]$, we have
\beq\label{d2est}
\begin{aligned}
-D_2 &\leq -(1-C\delta_1) \eta''(u_-) \frac{2}{(C\e^2 + \e f''(u_-))} \sum_{i=2}^3  \intOtil \frac{|w_{x_i}|^2}{z(1-z)}\,dz\,dx'\\
&\leq -(1-C\delta_1) \eta''(u_-) \frac{8}{(C\e^2 + \e f''(u_-))}  \sum_{i=2}^3  \intOtil |w_{x_i}|^2\,dz\,dx'\\
&=-(1-C\delta_1) \eta''(u_-) \frac{8}{C\e} \frac{\e^2}{\l^2} \sum_{i=2}^3  \intOtil |W_{x_i}|^2\,dz\,dx'\\
&\leq -(1-C\delta_1) \left(\frac{\eta''(u_-) f''(u_-)}{2}\frac{\e^3}{\l^2}\right)  \frac{1}{C\e^2}  \sum_{i=2}^3  \intOtil |W_{x_i}|^2\,dz\,dx'.
\end{aligned}
\eeq
Thus, we obtain
\[
-\frac{2}{\eta''(u_-) f''(u_-)} \frac{\l^2}{\e^3} D_2 \leq -\frac{1-C\delta_1}{C\e^2} \sum_{i=2}^3  \intOtil |W_{x_i}|^2\,dz\,dx'.
\]
$\bullet$ (Conclusion) Combining all these estimates, we have
\begin{align*}
&\frac{2}{\eta''(u_-) f''(u_-)} \frac{\l^2}{\e^3}  \mathcal{S}_{\e, \delta}(u)\\
&\leq -\frac{1}{C\delta_1} \left(\intOtil W^2\,dz\,dx' + 2\intOtil W\,dz\,dx' \right)^2 + (1+C\delta_1) \intOtil W^2\,dz\,dx' + \frac{2}{3}\intOtil W^3\,dz\,dx' \\
&\quad + C\delta_1 \intOtil |W|^3 dz dx' - (1-C\delta_1)\intOtil z(1-z)|\rd_z W|^2 dz dx' - \frac{1-C\delta_1}{C\e^2} \sum_{i=2}^3 \intOtil |W_{x_i}|^2\,dz\,dx'.
\end{align*}
Let $\delta_M$ be the constant in Lemma \ref{multiKV} corresponding to the constant $M$ in \eqref{M constant}. By taking $\delta_1$ and $\e$ small enough such that $C\delta_1 \leq \delta_M$ and $C\e^2 \leq \e^{3/2}$, we have
\begin{align*}
&\frac{2}{\eta''(u_-) f''(u_-)} \frac{\l^2}{\e^3}  \mathcal{S}_{\e, \delta}(u)\\
&\leq  -\frac{1}{\delta_M}\left(\intOtil W^2 \,dz\,dx' + 2\intOtil W \,dz\,dx'\right)^2 + (1+\delta_M) \intOtil W^2\,dz\,dx' + \frac{2}{3}\intOtil W^3\,dz\,dx'\\
&\quad  + \delta_M \intOtil |W|^3\,dz\,dx'- (1-\delta_M) \intOtil z(1-z)|\rd_z W|^2\,dz\,dx' - \frac{1-\delta_M}{\e^{3/2}} \sum_{i=2}^3 \intOtil |\rd_{x_i} W|^2\,dz\,dx'\\
&= R_{\e, \delta_M}(W),
\end{align*}
where $R_{\e, \delta_M}(W)$ is defined in Proposition \ref{multiKV}. Hence, by Proposition \ref{multiKV}, we conclude that
\[
\frac{2}{\eta''(u_-) f''(u_-)} \frac{\l^2}{\e^3}  \mathcal{S}_{\e, \delta}(u) \leq  R_{\e, \delta_M}(W) \leq 0.
\]
\end{proof}
\subsection{Truncation of the big values of $|\eta'(u) - \eta'(\util)|$}
Now, it remains to show that for values of $u$ such that $|\eta'(u) - \eta'(\util)|>\delta_1$, all bad terms can be controlled by the small portion of good terms. However, the value of $\d_1$ is conditioned to the constant $K$ in Proposition \ref{mainprop}. Therefore, we first need to find a uniform bound on $Y$ that is independent of the truncation size $\d_1$.

Now, we define a truncation function related to the value $|\eta'(u) - \eta'(\util)|$ with a truncation parameter $r>0$. For given $r>0$, let $\psi_r$ be a continuous function defined by
\[
\psi_r(y):= \begin{cases}
y &\text{if } |y| \leq r,\\
r &\text{if } y > r,\\
-r &\text{if } y < -r.\\
\end{cases}
\]
We define the truncation function $\ubar_r$ by
\begin{equation}\label{def ubar}
    \eta'(\ubar_r) - \eta'(\util) = \psi_r(\eta'(u) - \eta'(\util)).
\end{equation}
Note that for given $r>0$, $\ubar_r$ is well-defined, since $\eta'$ is one-to-one function.

First, we have the following lemma.
\begin{lemma}\label{energysmall}
    There exists $\delta_0, C, K > 0$ such that for any $\e, \l > 0$ with $\delta_0^{-1} \e < \l < \delta_0$, the following holds whenever $|Y(u)| \leq \e^2$.
    \[
    \intO a' \eta(u|\util)\,d\xi\,dx' \leq C\frac{\e^2}{\l},
    \]
    and
    \[
    |Y(\ubar_r)|\leq K\frac{\e^2}{\l} \, \, \text{for any} \, \, r>0. 
    \]
\end{lemma}
As we mentioned, the constant $K>0$ does not depend on the truncation size $r>0$.
\begin{proof}
Since the proof is essentially same as in \cite[Lemma 4.4]{Kang19}, we omit the proof.
\end{proof}
    From now on, we fix the constant $\delta_1 \leq \min(1, 2|u_-|)$ of Proposition \ref{inside} associated to the constant $K$ of Lemma \ref{energysmall}. For simplicity, we use the notation:
\[
\ubar := \ubar_{\delta_1}.
\]
Then, \eqref{def ubar} and Lemma \ref{energysmall} imply that
\[
|Y(\ubar)| \leq K \frac{\e^2}{\l} \, \, \text{and}\, \, |\eta'(\ubar) - \eta'(\util)| \leq \delta_1.
\]
Hence, thanks to the the Proposition \ref{inside}, we obtain
\begin{align*}
\mathcal{S}_{\e, \delta_1} (\ubar) &= -\frac{|Y(\ubar)|^2}{\e \delta_1} +   B_I(\ubar) + \delta_0 \frac{\e}{\l}\left|B_I(\ubar)\right| - \left(1-\delta_1 \frac{\e}{\l} \right)G_0(\ubar) - (1-\delta_1) D(\ubar)\\
& \leq 0.
\end{align*}
Recall that the good term $G$ in Lemma \ref{XYBG} can be decomposed into two good terms $G_0$ and $D$ in \eqref{BI}, that is, $G = G_0 + D$.
Now, we will compare values of good terms between the original variable $u$ and the truncated variable $\ubar$.
\begin{lemma}\label{G monotonicity}
Under the assumptions in Lemma \ref{energysmall}, the followings are true.
\begin{align*}
&0 \leq G_0(u) - G_0(\ubar) \leq G_0(u) \leq C\frac{\e^2}{\l},\\
&D_1(u) - D_1(\ubar) = \intO a \m (u) |\partial_\xi(\eta'(u) - \eta'(\ubar))|^2 \,d\xi\,dx',\\
&D_2(u) - D_2(\ubar) = \sum_{i=2}^3 \intO a \m (u) |\partial_{x_i}(\eta'(u) - \eta'(\ubar))|^2 \,d\xi\,dx'.
\end{align*}
Consequently, we have
\[
    0 \leq D_i(u) - D_i(\ubar) \leq D_i(u),\, \,   \forall i=1,2, \quad \text{and} \quad  0 \leq D(u) - D(\ubar) \leq D(u).
\]
\end{lemma}
\begin{proof}
Lemma \ref{G monotonicity} follows easily from the identity:
\begin{equation}\label{trunc identity}
    \begin{aligned}
|\eta'(u) - \eta'(\ubar)| &= |(\eta'(u) - \eta'(\util)) + (\eta'(\util) - \eta'(\ubar))|\\
&= |(\psi_{\delta_1} - I)(\eta'(u) - \eta'(\util))|\\
&= (|\eta'(u) - \eta'(\util)| - \delta_1)_+.
    \end{aligned}
\end{equation}
Thus, we skip the proof. (Details of the proof can be found in \cite{Kang19, KV21}.)
\end{proof}

For bad terms of $B$ in Lemma \ref{XYBG}, we define $B_O$ as 
\[
B_O := B - B_I, 
\]
so that $B$ is decomposed into the terms $B_I$ and $B_O$.
We now state the following proposition.
\begin{prop}\label{outside}
    There exists constants $\delta_0, C, C^* > 0$(In particular, $C$ depends on the constant $\delta_1$ in Proposition \ref{mainprop}) such that for any $\delta_0^{-1}\e < \l < \delta_0$, the following statements hold.
    \begin{enumerate}
    \item[1.] For any $u$ such that $|Y(u)| \leq \e^2$,
    \begin{align}
        \begin{split}\label{B out}
        &|B_I(u) - B_I(\ubar)| + |B_O(\ubar)| \leq C\delta_0^{1/2} D(u) + C\frac{\e}{\l}\delta_0 G_0(u),
        \end{split}\\
        \begin{split}\label{B total estimate}
        &|B(u)| \leq C\delta_0^{1/2} D(u) + C^* \frac{\e^2}{\l} .
        \end{split}
    \end{align}

    \item[2.] For any $u$ such that $|Y(u)| \leq \e^2$ and $D(u) \leq \frac{C^*}{4}\frac{\e^2}{\l}$,
    \begin{equation}\label{Y out}
    |Y(u) - Y(\ubar)|^2 \leq C\e \left(\d_0^2 D(u)  + \d_0^{5/2}\frac{\e}{\l}G_0(u) \right).
    \end{equation}
    \end{enumerate}
\end{prop}
To prove Proposition \ref{outside}, we first show the following estimates which are related to control our bad terms.
\begin{lemma}\label{outside lemma}
Under the same assumption as in Proposition \ref{outside}, for any $u$ such that $|Y(u)| \leq \e^2$, the followings hold:
\begin{align}
    \begin{split}\label{main outside estimate}
        &\intO a' |\eta'(u) - \eta'(\ubar)|^2\,d\xi\,dx' + \intO a' |\eta'(u) - \eta'(\ubar)|\,d\xi\,dx' \\
    & \quad \quad \leq C\e \l D_2(u) + C \sqrt{\frac{\e}{\l}} D_1(u) + C\left(\frac{\e}{\l}\right)^2 \intO a' \eta(u|\util)\,d\xi \, dx',
    \end{split}\\
    \begin{split}\label{lin bad term}
        \intO a' |u - \ubar|\,d\xi\,dx' + \leq C\e \l D_2(u) + C \sqrt{\frac{\e}{\l}} D_1(u) + C\left(\frac{\e}{\l}\right)^2 \intO a' \eta(u|\util)\,d\xi \, dx'.
    \end{split}
    \end{align}
\end{lemma} 
\begin{proof}
$\bullet$ Proof of \eqref{main outside estimate}:
First, observe that $(y-\delta_{1}/2)_+ \geq \delta_{1}/2$ whenever $(y-\delta_1)_+ > 0$. So, we have
\[
(y-\delta_1)_+  \leq (y-\delta_{1}/2)_+ \mathbf{1}_{\{y-\delta_1 > 0\}} \leq (y-\delta_{1}/2)_+  \left(\frac{(y-\delta_{1}/2)_+ }{\delta_1/2} \right) \leq \frac{2}{\delta_1}(y-\delta_{1}/2)_+ ^2.
\]
Taking $y = |\eta'(u) - \eta'(\util)|$, we have
\begin{equation}\label{Nonlinearization}
\begin{aligned}
    |\eta'(u) - \eta'(\ubar_{\delta_1})| &= (|\eta'(u) - \eta'(\util)|-\delta_1)_+ \\
    &\leq \frac{2}{\delta_1}(|\eta'(u) - \eta'(\util)|-\delta_1/2)_+^2\\
    &= \frac{2}{\delta_1}|\eta'(u) - \eta'(\ubar_{\delta_1/2})|^2 .
\end{aligned}
\end{equation}
Moreover, note that 
\begin{equation}\label{trunc mon}
\begin{aligned}
    |\eta'(u) - \eta'(\ubar_{\delta_1})| &= (|\eta'(u) - \eta'(\util)|-\delta_1)_+ \\
    &\leq (|\eta'(u) - \eta'(\util)|-\delta_1/2)_+\\
    &=|\eta'(u) - \eta'(\ubar_{\delta_1/2})|.
\end{aligned}
\end{equation}
Thus, in order to show \eqref{main outside estimate}, it is enough to consider the quadratic part with $\ubar$ defined with $\delta_1/2$ instead of $\delta_1$. We will keep this notation $\ubar = \ubar_{\delta_1/2}$ for this case below.\\
\textit{Step 1)} From Lemma \ref{shock property} and Lemma \ref{energysmall}, we have
\begin{align*}
2\e \int_{-1/\e}^{1/\e} \intT \eta(u|\util) \,dx'\,d\xi & \leq \frac{2\e}{\inf\limits_{\xi \in [-1/\e, 1/\e]} a'} \intRT a'\eta(u|\util)\,d\xi\,dx' \leq C \frac{\e}{\l\e} \frac{\e^2}{\l} = C\left( \frac{\e}{\l}\right)^2.
\end{align*}
By the mean value theorem, there exists $\xi_0 \in [-1/\e, 1/\e]$ such that
\[
\intT \eta(u|\util)(\xi_0) \,dx' \leq C \left(\frac{\e}{\l}\right)^2.
\]
Moreover, note that
\begin{equation}\label{shock unif bound}
    |\util| \leq |\util - u_-| + |u_-| \leq \e + |u_-| \leq \d_0 + |u_-| \leq 2|u_-| \quad \text{for } \d_0 \, \, \text{small enough}.
\end{equation}
Thanks to \eqref{shock unif bound}, we use the entropy hypothesis (i) of (A2) with $\theta = 2|u_-|$ which yields
\begin{align*}
&\intT |\eta'(u) - \eta'(\util)| (\xi_0)\,dx'\\
&\quad = \intT |\eta'(u) - \eta'(\util)|\mathbf{1}_{|u| \leq 4|u_-|} (\xi_0)\,dx' + \intT |\eta'(u) - \eta'(\util)|\mathbf{1}_{|u| > 4|u_-|} (\xi_0)\,dx'\\
&\quad \leq \sqrt{ \intT |\eta'(u) - \eta'(\util)|^2\mathbf{1}_{|u| \leq 4|u_-|} (\xi_0)\,dx'} + C\intT \eta(u|\util)(\xi_0) \,dx'\\
&\quad \leq C \sqrt{\intT \eta(u|\util)(\xi_0) \,dx'} + C\intT \eta(u|\util)(\xi_0) \,dx'\\
&\quad \leq C\frac{\e}{\l}.
\end{align*}
For convenience, we denote $\eta'(u) - \eta'(\util) =: \wtil$, and $\eta'(u) - \eta'(\ubar) =: \wbar$. Then, the above inequality is equivalent to
\[
\intT |\wtil|(\xi_0)\,dx' \leq C \frac{\e}{\l}.
\]
\textit{Step 2)} From \eqref{trunc identity}, we have
\[
|\wbar| = |\eta'(u) - \eta'(\ubar)| = (|\eta'(u) - \eta'(\util)| - \delta_1)_+ \leq |\eta'(u) - \eta'(\util)| = |\wtil|.
\]
This implies that
\begin{equation}\label{reference point estimate}
    \intT |\wbar|(\xi_0)\,dx' \leq \intT |\wtil|(\xi_0)\,dx' \leq C \frac{\e}{\l}.
\end{equation}
\textit{Step 3)} Now, we estimate $\displaystyle{\intRT a'\wbar^2 \, d\xi \, dx'}$. First, observe from Lemma \ref{shock property} that 
\begin{equation}\label{a der bound}
    \| a' \|_{L^\infty(\mathbb{R})}\leq C \l \e.
\end{equation}
Using \eqref{a der bound}, Lemma \ref{G monotonicity} and Poincar\'e inequality, we obtain
\beq\label{worstbad}
\begin{aligned}
\intO a'\wbar^2 \, d\xi \, dx' &= \intRT a'\wbar^2 \, d\xi \, dx'  - \intR a' \left(\intT \wbar \,dx' \right)^2\,d\xi +  \intR a' \left(\intT \wbar \,dx' \right)^2\,d\xi\\
&= \intR a' \left(\intT \wbar^2 \,dx' - \left(\intT \wbar \,dx' \right)^2\right)\,d\xi +  \intR a' \left(\intT \wbar \,dx' \right)^2\,d\xi\\
& \leq C \sum_{i=2}^3 \intR a' \intT |\wbar_{x_i}|^2\,dx' \,d\xi + \intR a' \left(\intT \wbar \,dx' \right)^2\,d\xi\\
& \leq C\e \l \sum_{i=2}^3 \intRT |\wbar_{x_i}|^2\,d\xi\,dx' +\intR a' \left(\intT \wbar \,dx' \right)^2\,d\xi\\
&\leq C\e \l D_2(u) + \intR a' \left(\intT \wbar \,dx' \right)^2\,d\xi.
\end{aligned}
\eeq
\textit{Step 4)} Now, we add the characteristic function $\mathbf{1}_{|\wbar|>0}$ into the integrand. Note that
\begin{align*}
&\intR a' \left(\intT \wbar \,dx' \right)^2\,d\xi = \intR a' \left(\intT \mathbf{1}_{|\wbar|>0} \wbar \,dx'\right)^2\,d\xi \\
&\qquad \leq \intR a' \left(\intT \mathbf{1}_{|\wbar|>0} \,dx'\right) \left(\intT \wbar^2 \,dx' \right)\,d\xi\\
&\qquad = \intR a' \left(\intT \mathbf{1}_{|\wbar|>0} \,dx'\right) \left(\intT \wbar^2 \,dx' - \left(\intT \wbar\,dx' \right)^2+ \left(\intT \wbar\,dx' \right)^2\right)\,d\xi.
\end{align*}
From \eqref{a der bound}, Lemma \ref{G monotonicity}, Poincar\'e inequality, and $\displaystyle{\intT \mathbf{1}_{|\wbar|>0}\,dx' \leq 1}$, we have
\begin{align*}
& \intR a' \left(\intT \wbar \,dx' \right)^2\,d\xi\\
&\leq C\intR a' \left(\intT \mathbf{1}_{|\wbar|>0} dx'\right)\left(\sum_{i=2}^3 \intT |\wbar_{x_i}|^2 dx'\right)d\xi  + \intR a' \left(\intT \mathbf{1}_{|\wbar|>0} dx'\right)\left(\intT \wbar dx'\right)^2 d\xi\\
& \leq C\e\l \sum_{i=2}^3 \intRT |\wbar_{x_i}|^2\,d\xi\,dx' + \intR a' \left(\intT \mathbf{1}_{|\wbar|>0}\,dx'\right) \left(\intT \wbar \,dx' \right)^2\,d\xi \\
& \leq  C\e \l D_2(u)  + \intR a' \left(\intT \mathbf{1}_{|\wbar|>0}\,dx'\right) \left(\intT \wbar \,dx' \right)^2\,d\xi .
\end{align*}
\textit{Step 5)} In this step, we estimate $\displaystyle{\intT \wbar \,dx'}$. For that, we use \eqref{reference point estimate} to have
\beq\label{pwout}
\begin{aligned}
\intT \wbar(\xi,x') \,dx' &= \intT \wbar(\xi,x') \,dx' - \intT \wbar(\xi_0,x') \,dx'  + \intT \wbar(\xi_0,x') \,dx'\\
&\leq \intT \int_{\xi_0}^{\xi} \wbar_{\zeta}(\zeta,x')\,d\zeta \,dx' + C\frac{\e}{\l}\\
&\leq \intT \sqrt{|\xi - \xi_0|}\sqrt{\intR |\wbar_{\zeta}(\zeta,x')|^2\,d\zeta} + C\frac{\e}{\l}\\
&\leq  \sqrt{|\xi| + \frac{1}{\e}}\intT \sqrt{\intR |\wbar_{\zeta}(\zeta,x')|^2\,d\zeta}\,dx' + C\frac{\e}{\l}.
\end{aligned}
\eeq
\textit{Step 6)} Thus, we obtain
\begin{align*}
&\intR a' \left(\intT \mathbf{1}_{|\wbar|>0}\,dx'\right) \left(\intT \wbar \,dx' \right)^2\,d\xi \\
&\quad \leq \intR a' \left(\intT \mathbf{1}_{|\wbar|>0}\,dx'\right) \left(2\left(|\xi|+ \frac{1}{\e}\right)\intT \intR |\wbar_{\zeta}(\zeta,x')|^2\,d\zeta\,dx'+2C\left(\frac{\e}{\l}\right)^2 \right)\\
&\quad \leq C D_1(u)\intRT a' \left(|\xi|+\frac{1}{\e}\right) \mathbf{1}_{|\wbar|>0}\,d\xi \,dx'  + C \left(\frac{\e}{\l}\right)^2 \intRT a' \mathbf{1}_{|\wbar|>0}\,d\xi \,dx'.
\end{align*}
Combining Step 1 through Step 6, we get
\begin{align*}
    &\intRT a' \wbar^2 \,d\xi\,dx' \\
    &\quad \leq C\e\l D_2(u) + CD_1(u)\intO a' \left(|\xi|+\frac{1}{\e}\right) \mathbf{1}_{|\wbar|>0}\,d\xi \,dx' +  C \left(\frac{\e}{\l}\right)^2 \intO a' \mathbf{1}_{|\wbar|>0}\,d\xi \,dx'.
\end{align*}
\textit{Step 7)} 
To control last two terms in the above inequality, we first observe that the definition of $\ubar$ implies
\[
|\big(\eta'(u) - \eta'(\util)\big)(\xi,x')| > \delta_1, \, \, \text{whenever } |\big(\eta'(u) - \eta'(\ubar)
\big)(\xi,x')| > 0.
\]
Thanks to \eqref{shock unif bound}, we use the entropy hypothesis (i) of (A2) with $\theta = 2|u_-|$, that is,
\[
|\big(\eta'(u) - \eta'(\util)\big)(\xi,x')|^2\mathbf{1}_{|u|\leq 4|u_-|} + |\big(\eta'(u) - \eta'(\util)\big)(\xi,x')|\mathbf{1}_{|u|\geq 4|u_-|}  \leq C \eta(u|\util),
\]
to have
\[
\eta(u|\util) \geq C \delta_1^2, \, \, \text{whenever } (\xi,x') \,\, \text{satisfying }  |\big(\eta'(u) - \eta'(\ubar)\big)(\xi,x')| > 0,
\]
which gives
\begin{equation}\label{ent lower bnd}
    \mathbf{1}_{|\wbar|>0} = \mathbf{1}_{\left|(\eta'(u) - \eta'(\ubar))(\xi,x')\right| > 0} \leq C \delta_1^{-2} \eta(u|\util).
\end{equation}
From \eqref{ent lower bnd}, we obtain
\[
\intRT a' \mathbf{1}_{|\wbar|>0}\,d\xi \,dx' \leq \frac{C}{\delta_1^2} \intRT a' \eta(u|\util)\,d\xi\,dx'.
\]
We now use Lemma \ref{shock property}, Lemma \ref{energysmall} and \eqref{ent lower bnd} to have
\begin{align*}
&\intRT a' \left(|\xi|+\frac{1}{\e}\right) \mathbf{1}_{|\wbar|>0}\,d\xi \,dx' \\
& \quad \leq C \frac{1}{\e}\sqrt{\frac{\l}{\e}} \intT \int_{|\xi| \leq \frac{1}{\e}\sqrt{\frac{\l}{\e}}} a' \mathbf{1}_{|\wbar|>0}\,d\xi \,dx' + C \intT \int_{|\xi| > \frac{1}{\e}\sqrt{\frac{\l}{\e}}} a'(\xi)|\xi| \mathbf{1}_{|\wbar|>0}\,d\xi \,dx' \\
&\quad \leq C\frac{1}{\e}\sqrt{\frac{\l}{\e}}\intRT a' \eta(u|\util)\,d\xi\,dx' + C  \intT \int_{|\xi| > \frac{1}{\e}\sqrt{\frac{\l}{\e}}} \e \l e^{-c\e|\xi|}|\xi|\,d\xi \,dx'\\
&\quad \leq C \sqrt{\frac{\e}{\l}} + C \frac{\l}{\e}\int_{|\xi| >\sqrt{\frac{\l}{\e}}} e^{-c|\xi|}|\xi| \,d\xi.
\end{align*}
As in \cite{Kang19}, taking $\d_0$ small enough such that $|\xi|\leq e^{\frac{c}{2}|\xi|}$ for $\xi \geq \sqrt{\frac{\l}{\e}}$ where $\frac{\e}{\l} \leq \d_0$, we have
\[
    \frac{\l}{\e}\int_{|\xi| >\sqrt{\frac{\l}{\e}}}  e^{-c|\xi|}|\xi| \,d\xi \leq \frac{\l}{\e}\int_{|\xi| >\sqrt{\frac{\l}{\e}}}e^{-\frac{c}{2}|\xi|} \,d\xi \leq \frac{2\l}{c\e}e^{-\frac{c}{2}\sqrt{\frac{\l}{\e}}} \leq \sqrt{\frac{\e}{\l}}.
\]
Thus, we obtain 
\[
    \intO a' \left(|\xi|+\frac{1}{\e}\right) \mathbf{1}_{|\wbar|>0}\,d\xi \,dx' \leq C\sqrt{\frac{\e}{\l}}.
\]
Hence, we get
\begin{equation}\label{sub main outside estimate}
    \begin{aligned}
        &\intO a' \wbar^2 \,d\xi\,dx'\\
&\quad \leq C\e\l D_2(u) + CD_1(u)\intO a' \left(|\xi|+\frac{1}{\e}\right) \mathbf{1}_{|\wbar|>0}\,d\xi \,dx' +  C \left(\frac{\e}{\l}\right)^2 \intO a' \mathbf{1}_{|\wbar|>0}\,d\xi \,dx'\\
&\quad \leq  C\e \l D_2(u) + C \sqrt{\frac{\e}{\l}} D_1(u) + C\left(\frac{\e}{\l}\right)^2 \intO a' \eta(u|\util)\,d\xi \, dx'.
    \end{aligned}
\end{equation}
Now, recall that $\ubar = \ubar_{\d_1/2}$ in the previous steps. From \eqref{sub main outside estimate}, \eqref{Nonlinearization} and \eqref{trunc mon}, we conclude that
\begin{align*}
    &\intO a' |\eta'(u) - \eta'(\ubar_{\d_1})|^2\,d\xi\,dx' + \intO a' |\eta'(u) - \eta'(\ubar_{\d_1})|\,d\xi\,dx' \\
    &\quad \leq C\intO a' |\eta'(u) - \eta'(\ubar_{\d_1/2})|^2\,d\xi\,dx' \\
    &\quad \leq C\e \l D_2(u) + C \sqrt{\frac{\e}{\l}} D_1(u) + C\left(\frac{\e}{\l}\right)^2 \intO a' \eta(u|\util)\,d\xi \, dx'.
\end{align*}
$\bullet$ Proof of \eqref{lin bad term}: Since $\eta'' \geq \alpha$, we have
\begin{align*}
\intO a' |u - \ubar_{\d_1}|\,d\xi\,dx' &\leq \frac{1}{\alpha} \intO a' |\eta'(u) - \eta'(\ubar_{\d_1})|\,d\xi\,dx'\\
&\leq  C\e \l D_2(u) + C \sqrt{\frac{\e}{\l}} D_1(u) + C\left(\frac{\e}{\l}\right)^2 \intO a' \eta(u|\util)\,d\xi \, dx'.
\end{align*}
This completes the proof.
\end{proof}

Now, we present the proof of Proposition \ref{outside}.\vspace{0.5cm}

$\bullet$ \textit{Proof of \eqref{B out}:}
Following the proof of \cite[Proposition 4.5]{Kang19}, we find that
\begin{align*}
    &|B_I(u) - B_I(\ubar)| + |B_O(u)| \\
    &\quad \leq C\left(\intO a' |\eta'(u) - \eta'(\ubar)|^2\,d\xi\,dx' + \intO a' |\eta'(u) - \eta'(\ubar)|\,d\xi\,dx' + \intO a' |u - \ubar|\,d\xi\,dx'\right).
\end{align*}
Therefore, Lemma \ref{outside lemma} implies that 
\[
    |B_I(u) - B_I(\ubar)| + |B_O(u)| \leq C \d_0^{1/2}D(u) + C\frac{\e}{\l}\d_0 G_0(u).
\]
$\bullet$ \textit{Proof of \eqref{B total estimate}:}
Now, we will give a rough estimate for $|B(u)|$. First, we use Taylor expansion to have
\begin{align*}
|B_I(\ubar)| &\leq C \intO a' |\ubar - \util|^2 \,d\xi\,dx'.
\end{align*}
Then, Lemma \ref{entropy ineq}, Lemma \ref{energysmall} and Lemma \ref{G monotonicity} yields that
\begin{equation}\label{BI total}
    \begin{aligned}
        |B_I(\ubar)| \leq C\intO a' \eta(\ubar | \util) \,d\xi\,dx' \leq C \intO a'\eta(u|\util) \,d\xi\,dx' \leq C\frac{\e^2}{\l}.
    \end{aligned}
\end{equation}
From Lemma \ref{G monotonicity}, \eqref{B out} and \eqref{BI total}, we obtain
\begin{align*}
    |B(u)| \leq C\delta_0^{1/2} D(u) + C\frac{\e}{\l}\delta_0 G_0(u) + C\frac{\e^2}{\l} \leq C\delta_0^{1/2} D(u) + C^*\frac{\e^2}{\l}.
\end{align*}
$\bullet$ \textit{Proof of \eqref{Y out}:}
Finally, it remains to estimate $|Y(u)- Y(\ubar)|$. Based on the proof of \cite[Propotision 4.5]{Kang19}, we find that
\begin{align*}
    |Y(u) - Y(\ubar)|\leq \intO a' |\eta'(u) - \eta'(\ubar)| \,d\xi\,dx' + C \intO a' |u-\ubar| \,d\xi\,dx'.
\end{align*}
From Lemma \ref{outside lemma}, we obtain
\[
    |Y(u) - Y(\ubar)| \leq C\left( \e \l D_2(u) +  \sqrt{\frac{\e}{\l}} D_1(u) + \left(\frac{\e}{\l}\right)^2 G_0(u) \right).
\]
Moreover, if $D(u) \leq \frac{C^*}{4}\frac{\e^2}{\l}$, then we have
\begin{align*}
|Y(u) - Y(\ubar)| \leq  C\left( \e \l \frac{\e^2}{\l} +  \sqrt{\frac{\e}{\l}} \frac{\e^2}{\l} + \left(\frac{\e}{\l}\right)^2 \frac{\e^2}{\l} \right) \leq C\frac{\e^2}{\l} \sqrt{\frac{\e}{\l}}.
\end{align*}
Hence, using $\frac{\e}{\l} \leq \d_0$, we conclude that
\begin{align*}
|Y(u) - Y(\ubar)|^2 &\leq C \frac{\e^2}{\l}\left(\frac{\e}{\l} D_1(u) + \sqrt{\frac{\e}{\l}} \e \l D_2(u) + \left(\frac{\e}{\l} \right)^{5/2}G_0(u) \right)\\
&\leq C\e \left(\left(\frac{\e}{\l}\right)^2 D(u)  + \left(\frac{\e}{\l} \right)^{5/2}\frac{\e}{\l}G_0(u) \right)\\
&\leq C\e \left(\d_0^2 D(u)  + \d_0^{5/2}\frac{\e}{\l}G_0(u) \right).
\end{align*}
This completes the proof of Proposition \ref{outside}.\\
\qed
\subsection{Conclusion}
We divide into two cases: whether the diffusion term $D(u)$ is big or small.\\
\textit{Case I:} $D(u) \geq 4 C^*  \frac{\e^2}{\l}$, where the constant $C^*$ is defined in Proposition \ref{outside}. Then, taking $\delta_0$ small enough, we have
\begin{align*}
R(u) &:= -\frac{1}{\e^4}Y^2(u) + B(u) + \delta_0 \frac{\e}{\l}|B(u)| - G_0(u) - (1-\d_0)D(u)\\
&\leq 2|B(u)| - (1-\d_0) D(u) \leq 2C^*\frac{\e^2}{\l} - (1 - \d_0 - 2\delta_0^{1/4})D(u)\\
&\leq 2C^*\frac{\e^2}{\l} - \frac{1}{2}D(u)\\
&\leq 0.
\end{align*}
\textit{Case II:} $D(u) \leq 4C^*\frac{\e^2}{\l}.$\\
Since
\[
|Y(\ubar)|^2 \leq 2(|Y(u)|^2 + |Y(u) - Y(\ubar)|^2),
\]
we have
\[
-2|Y(u)|^2  \leq  -|Y(\ubar)|^2  + 2|Y(u) - Y(\ubar)|^2.
\]
Using Lemma \ref{G monotonicity} and $\e <\d_1$, we obtain
\begin{align*}
R(u) &\leq -\frac{2}{\e \delta_1}|Y(u)|^2 + B(u) + \delta_0 \frac{\e}{\l}|B(u)| - G_0(u) - (1-\d_0)D(u)\\
&\leq -\frac{|Y(\ubar)|^2}{\e \delta_1} +  B_I(\ubar) + \delta_0 \frac{\e}{\l}\left|B_I(\ubar)\right| - \left(1-\delta_1 \frac{\e}{\l} \right)G_0(\ubar) - (1-\delta_1) D(\ubar)\\
& \quad + \underbrace{\frac{2}{\e \delta_1} |Y(u) - Y(\ubar)|^2}_{=:J_1} + \underbrace{\left(1+\delta_0 \frac{\e}{\l}\right) \left(|B_I(u) - B_I(\ubar)| + |B_O(u)| \right)}_{=J_2}\\
&\quad -\delta_1 \frac{\e}{\l} G_0(u) - (\delta_1 - \d_0) D(u).
\end{align*}
We claim that
\[
J_1 + J_2 \leq \delta_1 \frac{\e}{\l} G_0(u) + \frac{\d_1}{2} D(u).
\]
Indeed, it follows from Proposition \ref{outside} that if we choose $\delta_0$ small enough such that $\delta_0 < \delta_1^{4}$, and for any $\e/\l < \delta_0$, we have
\begin{align*}
J_1 &= \frac{2}{\e \delta_1} |Y(u) - Y(\ubar)|^2 \leq \frac{C}{\d_1}\left(\d_0^2 D(u)  + \d_0^{5/2}\frac{\e}{\l}G_0(u) \right) \leq \frac{\d_1}{4} D(u) + \frac{1}{2}\delta_1\frac{\e}{\l}G_0(u),
\end{align*}
and
\begin{align*}
J_2 
\leq C\left(\delta_0^{1/2} D(u) + C\frac{\e}{\l}\delta_0 G_0(u)\right) \leq \frac{\d_1}{4} D(u) + \frac{1}{2}\delta_1\frac{\e}{\l}G_0(u).
\end{align*}
Thus, we have
\[
J_1 + J_2 \leq \delta_1 \frac{\e}{\l} G_0(u) + \frac{\d_1}{2} D(u).
\]
If we choose $\d_0$ small enough such that $ \d_0 <\frac{\d_1}{2}$, we obtain 
\[
    J_1 + J_2 \leq \delta_1 \frac{\e}{\l} G_0(u) + (\d_1 - \d_0) D(u).
\]
Hence, thanks to Proposition \ref{mainprop}, we conclude that
\begin{align*}
R(u) &\leq -\frac{|Y(\ubar)|^2}{\e \delta_1} +  B_I(\ubar) + \delta_0 \frac{\e}{\l}\left|B_I(\ubar)\right| - \left(1-\delta_1 \frac{\e}{\l} \right)G_0(\ubar) - (1-\delta_1) D(\ubar)\\
&\leq \mathcal{S}_{\e, \delta_1}(\ubar) \leq 0.
\end{align*}
This completes the proof of Proposition \ref{mainprop}.\\
\qed

\appendix
\setcounter{equation}{0}
\section{Proof of Lemma \ref{algebraic}}\label{proof alg}
Here, we present the proof of Lemma \ref{algebraic}. For that, we first introduce the simple inequality on a specific polynomial functional.
\begin{lemma}\label{polynomial ineq}
For all $x \in [-2,0)$,
\[
g(x) := 2x - 2x^2 - \frac{4}{3}x^3 +\frac{7}{5}\theta \big(-x^2 -2x \big)^{3/2} <0,
\]
where $\displaystyle{\theta = \sqrt{5 - \frac{\pi^2}{3}}} \approx 1.308$.
\end{lemma}
\begin{proof}
The proof of Lemma \ref{polynomial ineq} is similar to \cite[Lemma A.1]{KV21}.\\
\textit{Step 1) }In step 1, we will show that $g''(x)> 0$ on $x \in \left[-2, -\frac{1+\sqrt{3}}{2}\right]$.
Note that 
\[
-x^2 - 2x = 1 - (1+x)^2.
\]
Then, we have 
\begin{align*}
    \frac{g'(x)}{2} &= 1-2x - 2x^2 - 2.1 \theta (1+x)\sqrt{1- (1+x)^2},\\
    \frac{g''(x)}{2} &= 2 - 4x - 4.2 \theta \sqrt{1- (1+x)^2} + \frac{2.1 \theta}{\sqrt{1- (1+x)^2}}.
\end{align*}
Since $1-(1+x)^2$ is increasing on $ \left[-2, -\frac{1+\sqrt{3}}{2}\right]$, we have 
\[
    1-(1+x)^2 \geq \frac{\sqrt{3}}{2} \quad \text{on } x \in \left[-2, -\frac{1+\sqrt{3}}{2}\right].
\]
This implies that 
\[
    \frac{g''(x)}{2} \geq -2 - 4x - 4.2 \theta \sqrt{\frac{\sqrt{3}}{2}} + 2.1\theta \sqrt{\frac{2}{\sqrt{3}}}.
\]
But we have 
\[
    - 4.2 \theta \sqrt{\frac{\sqrt{3}}{2}} + 2.1\theta \sqrt{\frac{2}{\sqrt{3}}} \approx -2.16 > -2.2, \, \,  \text{and }   -\frac{1+\sqrt{3}}{2} < -\frac{4.2}{4}.
\]
Therefore, we obtain 
\[
    \frac{g''(x)}{2} > -4.2 - 4x > 0, \, \,  \text{whenever } -2 \leq x \leq - \frac{1+\sqrt{3}}{2}.
\]
\textit{Step 2) }In Step 2, we will show that 
\begin{equation}\label{poly g monotone}
    g'(x) >0 \, \, \text{on } x \in \left(-\frac{1+\sqrt{3}}{2}, -1 \right].
\end{equation}
First, note that 
\[
g'(x) = \underbrace{2 - 4x - 4x^2}_{:=h_1(x)} - \underbrace{4.2\theta (x+1)\sqrt{1-(1+x)^2}}_{=:h_2(x)}.
\]
We can easily see that $h_2(x) \leq 0$ on $(-2,-1)$. Since $-\frac{1\pm\sqrt{3}}{2}$ are the two roots of $h_1$, we deduce that \eqref{poly g monotone} holds.\\
\textit{Step 3) } Here, we will show that $g'(x) > 0$ on $\left(-1, \frac{-3+\sqrt{5}}{2}\right)$. Note that 
\[
g'(x) = -4(1+x)^2 + 4(1+x) + 2 - 4.2\theta (1+x)\sqrt{1-(1+x)^2}.
\]
If we denote $A := x+1$, then we can rewrite the above equation into
\[
g'(A) = -4A^2 + 4A + 2 - 4.2 \theta A \sqrt{1-A^2} \, \, \text{on } A \in \left(0, \frac{-1 + \sqrt{5}}{2}\right).
\]
We now claim that 
\[
h(A) := \frac{g'(A)}{A} = -4A + 4 + \frac{2}{A} - 4.2\theta \sqrt{1-A^2} > 0 \, \, \text{on } A \in \left(0, \frac{-1 + \sqrt{5}}{2}\right).
\]
Indeed, if $h$ attains its minimum at $A_* \in \left(0, \frac{-1 + \sqrt{5}}{2}\right)$, then we have 
\[
h'(A_*) = -4-\frac{2}{A_*^2} + 4.2\theta \frac{ A_*}{\sqrt{1-A_*^2}} = 0,
\]
which implies 
\[
\sqrt{1-A_*^2} = 4.2\theta \frac{A_*^3}{4A_*^2+2}.
\]
Then, we obtain 
\begin{align*}
    h(A_*) &= -4A_* +  4 + \frac{2}{A_*} - (4.2\theta)^2   \frac{A^3}{4A_*^2+2}\\
    &\geq   -4A_* +  4 + \frac{2}{A_*}- (4.2\theta)^2  \frac{A_*}{4}.
\end{align*}
Note that the function $h_1(A):= -4A +  4 + \frac{2}{A}- (4.2\theta)^2  \frac{A}{4}$ is decreasing on the interval $\left(0, \frac{-1 + \sqrt{5}}{2}\right)$. Therefore, we obtain 
\[
h(A_*) \geq h_1\left( \frac{-1 + \sqrt{5}}{2} \right) \approx 0.1 > 0.
\]
Since $h(A) \to +\infty$ as $A \to 0+$, and $h\left( \frac{-1 + \sqrt{5}}{2} \right) \approx 1.36 >0$, we conclude that $h(A) > 0$ on the interval $\left(0, \frac{-1 + \sqrt{5}}{2}\right)$.\\
\textit{Step 4) } Now, we prove our lemma. Since $g(x)$ is continuous on $[-2,0]$, so it attains maximum on $[-2,0]$.

Suppose that the global maximum is achieved at $x_* \in (-2,0)$. At this point, $g$ satisfies $g'(x_*)=0$ and $g''(x_*) \leq 0$. From Step 1 to Step 3, we have $x_* \in \left( \frac{-3+\sqrt{5}}{2}, 0\right)$. Since $g(0)=0$, and $g(x_*)$ is supposed to be a global maximum, we have $g(x_*)\geq 0$.

On the other hand, since $g'(x_*) = 0$, we get
\[
    g'(x_*) = 2-4x_* - 4x_*^2 - \frac{21}{5} \theta (1+x_*)\sqrt{1- (1+x_*)^2} = 0,
\]
which implies 
\begin{equation}\label{g argmax}
    \sqrt{1- (1+x_*)^2} = \frac{5}{21\theta} \frac{2-4x_* - 4x_*^2}{1+x_*}.
\end{equation}
From \eqref{g argmax}, we obtain
\begin{align*}
    g(x_*) &=  2x_* - 2x_*^2 - \frac{4}{3}x_*^3 +\frac{7}{5}\theta \big(-x_*^2 -2x_* \big)^{1/2}(-x_*^2 -2x_*)\\
    &=2x_* - 2x_*^2 - \frac{4}{3}x_*^3 + \frac{(2-4x_* - 4x_*^2)(-x_*^2 -2x_*)}{3(1+x_*)}\\
    &= \frac{2x_*(x_*^2+3x_*+1)}{3(1+x_*)}.
\end{align*}
Since
\[
    \frac{2x(x^2+3x+1)}{3(1+x)} < 0 \, \, \text{on } x \in \left( \frac{-3+\sqrt{5}}{2}, 0\right),
\]
we have $g(x_*)<0$, which yields contradiction. 

Therefore, $g$ attains its global maximum only at $x=0$ or $x=-2$. Since $g(-2) = -4/3$ and $g(0)=0$, we conclude that 
\[
g(x) < 0 \, \, \text{on } x \in [-2,0).
\]
\end{proof}
Now, we prove Lemma \ref{algebraic}. \vspace{0.3cm}

\textit{Proof of Lemma \ref{algebraic}:} As in \cite{KV21}, we split the proof into three steps.\\
\textit{Step 1)} For $r>0$, we denote $B_r(0)$ the open ball centered at the origin with radius $r$. In this step, we will show the following:

There exist $r>0$ and $\d_2>0$ such that for any $\d < \d_2$,
\begin{equation}\label{alg claim1}
    P_\d (Z_1, Z_2) - |E(Z_1, Z_2)|^2 \leq 0, \, \, \text{whenever } (Z_1, Z_2) \in B_r(0).
\end{equation}
To verify this claim, notice first that $|Z_1|, |Z_2| \leq r$ on $B_r(0)$. This implies that 
\[
|2Z_1|^2 = (E - (Z_1^2+ Z_2^2))^2 \leq 2|E|^2 + 2|Z_1^2 + Z_2^2|^2 \leq 2|E|^2 + 2r^2(Z_1^2 + Z_2^2).
\]
Therefore, we have 
\[
-|E|^2 \leq -2Z_1^2 + r^2(Z_1^2 + Z_2^2).
\]
Thus, for any $(Z_1, Z_2) \in B_r(0)$,
\begin{align*}
    & P_\d (Z_1, Z_2) - |E(Z_1, Z_2)|^2 \\
    &\quad \leq -2Z_1^2 + (1+\d)\left(Z_1^2 + Z_2^2 + \frac{r^2}{1+\d}(Z_1^2 + Z_2^2) + \frac{(2+6\d)r}{1+\d}Z_2^2 + \frac{((2/3) + 6\d)r}{1+\d}Z_1^2\right)\\
    &\qquad -2\left(1-\d - \left(\frac{2}{3} + \d\right)\frac{21}{20}\theta r\right)Z_2^2.
\end{align*}
Taking $\d_2$ and $r$ small enough, we obtain that for any $\d < \d_0$,
\[
    P_\d (Z_1, Z_2) - |E(Z_1, Z_2)|^2\leq 0 \, \, \text{on } B_r(0).
\]
\textit{Step 2)} In Step 2, we will prove the following claim:

There exists $\d_2>0$ and $\d_3>0$ such that for any $0<\d <\d_0$, we have
\begin{equation}\label{alg claim2}
    P_\d(Z_1, Z_2)<0, \, \, \text{whenever } |E(Z_1,Z_2)|\leq \d_3, \, \, \text{and } (Z_1, Z_2) \notin B_r(0).
\end{equation}
To prove the claim, we first observe that the limiting case: $\d=0$ and $E(Z_1,Z_2)=0$.

If $\d=0$ and $E(Z_1,Z_2)=0$, then we have 
\[
P_0(Z_1, Z_2) = 2Z_1 - 2Z_1^2 - \frac{4}{3}Z_1^3 + \frac{7}{5}\theta \big(-Z_1^2 - 2Z_1\big)^{3/2}, \quad Z_1^2 + Z_2^2 + 2Z_1 = 0.
\]
Since $(Z_1+1)^2 + Z_2^2=1$ by $E=0$, we have $-2 \leq Z_1 \leq 0$. By Lemma \ref{polynomial ineq}, we obtain 
\[
P_0(Z_1, Z_2)<0, \quad Z_1^2 + Z_2^2 + 2Z_1 = 0, \quad Z_1 \neq 0.
\]
Since $P_0$ is continuous, it attains its maximum $-c<0$ on the compact set $\{E(Z_1, Z_2)=0\} \setminus B_r(0)$. In addition, $P_0$ is uniformly continuous on the compact set $\{E(Z_1, Z_2) \leq 1\} \setminus B_r(0)$, so there exist $0<\d_3<1$ such that 
\[
P_0(Z_1, Z_2) < -c/2, \, \, \text{whenever } |E(Z_1, Z_2)| \leq \d_3, \, \, \text{and } (Z_1,Z_2) \notin B_r(0).
\]
Taking $\d_2$ small enough, for any $\d<\d_2$, we get 
\[
    P_\d(Z_1, Z_2) < 0, \, \, \text{whenever } |E(Z_1, Z_2)| \leq \d_3, \, \, \text{and } (Z_1,Z_2) \notin B_r(0).
\]
This proves our claim.\\
\textit{Step 3)} Hence, from \eqref{alg claim1} and \eqref{alg claim2}, we obtain \eqref{algebraic inequality}.\\
\qed
\section{Proof of Theorem \ref{timedecay}}\label{appendix B}
\setcounter{equation}{0}
Here, we present the proof of $L^\infty$-bound of the shift function $X(t)$ and $L^2$-time decay rate of the perturbation $u^X - \util$, under the condition that the initial perturbation satisfies $u_0 - \util \in L^1 \cap L^\infty(\Omega)$. In what follows, let $C_0$ be a positive constant that depends on $\|u_0\|_{L^\infty(\Omega)}$, and may change from line to line.

Now, we assume that
\[
    u_0 - \util \in L^1 \cap L^\infty(\Omega).
\]
\subsection{$L^\infty$ bound of shift $X(t)$}
In this subsection, we show the shift function $X(t)$ is bounded. Since the proof of \eqref{L inf shift} is almost identical to \cite[Section 3]{KO24}, we will give a sketch for the proof of \eqref{L inf shift}. 

First, we state the $L^1$-contraction property of scalar viscous conservation law, which is a well-known result.
\begin{lemma}\label{L1contraction}
    \cite[Theorem 6.3.2]{dafermos2005hyperbolic} Let $u$ and $v$ be solutions to
    \[
    u_t +  \div_x F(u) = \Delta u,
    \]
    with respective initial data $u_0$ and $v_0$ in $L^\infty(\Omega)$. Assume that the flux $F$ is $C^1$. If the initial data $u_0$ and $v_0$ satisfy $u_0 - v_0 \in L^1 \cap L^\infty (\Omega)$, we have the $L^1$ contraction
    \[
    \| u-v \|_{L^1(\Omega)} \leq \| u_0 - v_0 \|_{L^1(\Omega)}.
    \]
    \end{lemma}
    $\bullet$ \textit{Proof of \eqref{L inf shift}}: Following the calculations in \cite[Section 3]{KO24}, we have
\[
    |X(t)| \leq C\left(\int_\Omega |(u^X - \tiu(x_1 - \sigma t))|^2\,dx + 4 \|\tiu\|_\infty \int_\Omega |u^X- \tiu^X(x_1 - \sigma t)|\,dx\right)+1.
\]
From \eqref{L2 type contraction} and Lemma \ref{L1contraction}, we obtain
\begin{align*}
    |X(t)| \leq C\big(C_0 \| u_0 - \tiu \|_{L^2(\Omega)}^2 + 4 \max\{|u_-|,|u_+|\} \| u_0 - \tiu \|_{L^1(\Omega)}\big)+1, \quad \forall t>0.
\end{align*}
For convenience, we put
\begin{equation}\label{C* def}
    C_* := \max\{C_0,1\} \cdot \big(1+\| u_0 - \tiu \|_{L^2(\Omega)}^2  +\| u_0 - \tiu \|_{L^1(\Omega)}\big).
\end{equation}
Thus, we obtain
\begin{equation}\label{X bound conclusion}
    \|X\|_{L^\infty(\RR_+)} \leq C C_*
\end{equation}
\subsection{Proof of time decay estimate}
Now, we prove the time decay estimate \eqref{time estimate}. As in \cite{KO24}, we use the Gagliardo-Nirenberg type interpolation in $\Omega = \RR \times \mathbb{T}^2$ which was proved by Huang and Yuan \cite{Huang21}.

First, recall from \eqref{cont conclusion} that we have 
\begin{equation}\label{modified contraction}
        \frac{d}{dt} \intO a \eta(u^X|\util)\,dx  + \d_0 D(u) \leq 0.
    \end{equation}
Now, we use the following lemma.
\begin{lemma} \label{Gn interpolation}
    \cite[Theorem 1.4]{Huang21}
    (Gagliardo Nirenberg type inquality in $\Omega := \bbr \times \mathbb{T}^{2}$)
    Let $f \in L^1(\Omega)$ and $\nabla f \in L^2(\Omega)$, and $f$ is periodic in the $x_i$ direction for $i=2,3$. Then, it holds that
    \begin{equation}\label{GN}
    \|f\|_{L^2(\Omega)} \leq \sum_{k=0}^{2} \| \nabla f \|_{L^2(\Omega)}^{\theta_k} \| f \|_{L^1(\Omega)}^{1-\theta_k},
    \end{equation}
    where $\theta_k = \frac{k+1}{k+3}$, and the constant $C>0$ is independent of $f$.
    \end{lemma}
$\bullet$ \textit{$L^1$ bound of perturbation:} To apply the interpolation inequality \eqref{GN}, we first need to find the $L^1$-bound of perturbation $\eta'(u^X)-\eta'(\util)$. Note that 
\begin{align*}
    \|\eta'(u^X)-\eta'(\util)\|_{L^1(\Omega)} \leq &\|\eta'(u^X)-\eta'(\util^X)\|_{L^1(\Omega)} + \|\eta'(\util^X)-\eta'(\util)\|_{L^1(\Omega)} =: I_1 + I_2.
\end{align*}
For $I_1$, we use Lemma \ref{L1contraction} to have 
\begin{align*}
    I_1 = \|\eta'(u)-\eta'(\util)\|_{L^1(\Omega)}\leq C_0 \|u-\util(\cdot - \s t)\|_{L^1(\Omega)} \leq   C_0 \|u_0-\util\|_{L^1(\Omega)}.
\end{align*}
To estimate $I_2$, observe that (see \cite[Section 3]{Kang-V-1} for details)
\[
I_2 = \|\eta'(\util^X)-\eta'(\util)\|_{L^1(\Omega)} \leq C_0 \|\util^X-\util\|_{L^1(\Omega)} = C_0 |X(t)|(u_- - u_+).
\]
Using \eqref{C* def} and \eqref{X bound conclusion}, we get the $L^1$ estimate:
\begin{equation}\label{L1 conclusion}
    \|\eta'(u^X)-\eta'(\util)\|_{L^1(\Omega)} \leq I_1 + I_2 \leq C C_*.
\end{equation}
$\bullet$ \textit{Decay estimate for perturbation:} Using the interpolation inequality \eqref{GN} and the $L^1$ bound \eqref{L1 conclusion}, we have
\begin{align*}
    &\|\eta'(u^X) - \eta'(\tiu(\cdot - \sigma t))\|_{L^2(\Omega)}\\
    &\leq C\sum_{k=0}^{2} \|\nabla(\eta'(u^X) - \eta'(\tiu(\cdot - \sigma t)))\|_{L^2(\Omega)}^{\frac{k+1}{k+3}} \|\eta'(u^X) - \eta'(\tiu(\cdot - \sigma t))\|_{L^1(\Omega)}^{1-\frac{k+1}{k+3}}\\
    & \leq C C_* \sum_{k=0}^{2}(C_*^{-1}\|\nabla(\eta'(u^X) - \eta'(\tiu(\cdot - \sigma t)))\|_{L^2(\Omega)})^{\frac{k+1}{k+3}}.
    \end{align*}
    As in \cite{KO24}, if we denote $B := C_*^{-1}\|\nabla(u^X - \tiu(\cdot - \sigma t))\|_{L^2(\Omega)}$, the above inequality can be written as follows:
    \[
    \|\eta'(u^X) - \eta'(\tiu(\cdot - \sigma t))\|_{L^2(\Omega)} \leq CC_* (B^{1/3} + B^{2/4} + B^{3/5}).
    \]
    Now, we claim that
    \[
    \|\eta'(u^X) - \eta'(\tiu(\cdot - \sigma t))\|_{L^2(\Omega)} \leq 3C C_* B^{1/3}.
    \]
    To prove the claim, we split into two cases: $B \leq 1$ and $B \geq 1$.
    \begin{enumerate}
    \item[(i)] If $B \leq 1$, we have 
    \[
    \|\eta'(u^X) - \eta'(\tiu(\cdot - \sigma t))\|_{L^2(\Omega)} \leq 3 C C_0 B^{1/3}.
    \]
    \item[(ii)] Now, suppose that $B>1$. Then, we obtain the following upper bound
    \[
    \|u^X - \tiu(\cdot - \sigma t)\|_{L^2(\Omega)} \leq 3C C_*  B^{3/5}.
    \]
    On the other hand, the $L^2$-contraction estimate \eqref{L2 type contraction} implies that 
    \[
    \|\eta'(u^X) - \eta'(\tiu(\cdot - \sigma t))\|_{L^2(\Omega)} \leq C_0  \| u_0 - \tiu \|_{L^2(\Omega)} \leq \sqrt{C_0} \sqrt{C_*} \leq C_*.
    \]
    Therefore, we have
    \begin{align*}
    &\|\eta'(u^X) - \eta'(\tiu(\cdot - \sigma t))\|_{L^2(\Omega)}\\
    &= \|\eta'(u^X) - \eta'(\tiu(\cdot - \sigma t))\|_{L^2(\Omega)}^{\frac{5}{9}} \|\eta'(u^X) - \eta'(\tiu(\cdot - \sigma t))\|_{L^2(\Omega)}^{1-\frac{5}{9}} \\
    &\leq (3 C C_* B^{\frac{3}{5}})^{\frac{5}{9}} C_*^{1-\frac{5}{9}} \leq 3C C_* B^{1/3}.
    \end{align*}
    \end{enumerate}
    Thus, we get
    \begin{equation*}
    \begin{aligned}
        \| \eta'(u^X) - \eta'(\tiu(\cdot - \sigma t)) \|_{L^2(\Omega)} &\leq 3C C_0  B^{1/3}= 3C C_*^{2/3}\|\nabla(\eta'(u^X) - \eta'(\tiu(\cdot - \sigma t)))\|_{L^2(\Omega)}^{\frac{1}{3}}.
    \end{aligned}
    \end{equation*}
    This implies that
    \[
    \| \eta'(u^X) - \eta'(\tiu(\cdot - \sigma t))\|_{L^2(\Omega)}^6 \leq C C_*^4 \|\nabla( \eta'(u^X) - \eta'(\tiu(\cdot - \sigma t)))\|_{L^2(\Omega)}^{2}.
    \]
    Thus, by the contraction estimate \eqref{modified contraction}, we obtain
    \begin{align*}
        \frac{d}{dt} \intO a \eta(u^X|\util)\,dx  &\leq -\d_0 \intO |\nb ( \eta'(u^X) - \eta'(\tiu(\cdot - \sigma t)))|^2\,dx\\
        &\leq  -\frac{\d_0}{C C_*^4} \| \eta'(u^X) - \eta'(\tiu(\cdot - \sigma t))\|_{L^2(\Omega)}^6.
    \end{align*}
    By the maximum principle, note that 
    \[
        \eta(u^X|\util) \leq C_0|\eta'(u^X) - \eta'(\util)|^2  \leq C_*|\eta'(u^X) - \eta'(\util)|^2.
    \]
    Thus, we have 
    \begin{equation}\label{ent gronwall}
        \frac{d}{dt} \intO a \eta(u^X|\util(\cdot - \s t))\,dx \leq - \frac{C}{C_*^3}\d_0 \left( \intO a \eta(u^X|\util(\cdot - \s t))\,dx \right)^3.
    \end{equation}
    Hence, from Lemma \ref{entropy ineq}, \eqref{ent gronwall} and the inequality $2(x+y)^{1/4} \geq x^{1/4} + y^{1/4}$, we conclude that
    \begin{align*}
    \|u^X - \tiu(\cdot - \sigma t)\|_{L^2(\Omega)} &\leq C\sqrt{\intO a \eta(u^X|\util(\cdot - \s t))dx} \leq C\left(\frac{C_*^3\left( \intO a \eta(u_0|\util)dx \right)^2}{2C\d_0\left( \intO a \eta(u_0|\util)dx \right)^2 t + C_*^3} \right)^\frac{1}{4}\\
    & \leq \frac{2C C_*^{3/4}\left( \intO a \eta(u_0|\util)\,dx \right)^{1/2}}{(2C\d_0)^{1/4}\left( \intO a \eta(u_0|\util)\,dx \right)^{1/2} t^{1/4} + C_*^{3/4}}.
    \end{align*}
    This completes the proof of \eqref{time estimate}.
    \qed


 
\vspace{0.5cm}
\bibliography{reference.bib}

\begin{thebibliography}{10}

\bibitem{Bressan}
{\sc Bressan, A.}
\newblock {\em Hyperbolic systems of conservation laws: the one-dimensional {C}auchy problem}.
\newblock Oxford University Press, Oxford, 2000.

\bibitem{dafermos2005hyperbolic}
{\sc Dafermos, C.}
\newblock {\em Hyperbolic conservation laws in continuum physics}, vol.~3.
\newblock Springer, 2005.

\bibitem{F-S}
{\sc Freist\"uhler, H., and Serre, D.}
\newblock {$\mathscr L^1$} stability of shock waves in scalar viscous conservation laws.
\newblock {\em Comm. Pure Appl. Math. 51}, 3 (1998), 291--301.

\bibitem{G89}
{\sc Goodman, J.}
\newblock Stability of viscous scalar shock fronts in several dimensions.
\newblock {\em Trans. Amer. Math. Soc. 311}, 2 (1989), 683--695.

\bibitem{HZ00}
{\sc Hoff, D., and Zumbrun, K.}
\newblock Asymptotic behavior of multidimensional scalar viscous shock fronts.
\newblock {\em Indiana Univ. Math. J. 49}, 2 (2000), 427--474.

\bibitem{HZ02}
{\sc Hoff, D., and Zumbrun, K.}
\newblock Pointwise {G}reen's function bounds for multidimensional scalar viscous shock fronts.
\newblock {\em J. Differential Equations 183}, 2 (2002), 368--408.

\bibitem{Huang21}
{\sc Huang, F., and Yuan, Q.}
\newblock Stability of planar rarefaction waves for scalar viscous conservation law under periodic perturbations.
\newblock {\em Methods Appl. Anal. 28}, 3 (2021), 337--353.

\bibitem{Kang19}
{\sc Kang, M.-J.}
\newblock {$L^2$}-type contraction for shocks of scalar viscous conservation laws with strictly convex flux.
\newblock {\em J. Math. Pures Appl. (9) 145\/} (2021), 1--43.

\bibitem{KO24}
{\sc Kang, M.-J., and Oh, H.}
\newblock $l^2$ decay for large perturbations of viscous shocks for multi-d burgers equation.
\newblock {\em arXiv:2403.08445, To appear in Anal. Appl.\/} (2024).

\bibitem{KV16}
{\sc Kang, M.-J., and Vasseur, A.~F.}
\newblock Criteria on contractions for entropic discontinuities of systems of conservation laws.
\newblock {\em Arch. Ration. Mech. Anal. 222}, 1 (2016), 343--391.

\bibitem{Kang-V-1}
{\sc Kang, M.-J., and Vasseur, A.~F.}
\newblock {$L^2$}-contraction for shock waves of scalar viscous conservation laws.
\newblock {\em Ann. Inst. H. Poincar\'{e} C Anal. Non Lin\'{e}aire 34}, 1 (2017), 139--156.

\bibitem{KV21}
{\sc Kang, M.-J., and Vasseur, A.~F.}
\newblock Contraction property for large perturbations of shocks of the barotropic {N}avier-{S}tokes system.
\newblock {\em J. Eur. Math. Soc. 23}, 2 (2021), 585--638.

\bibitem{KV-Inven}
{\sc Kang, M.-J., and Vasseur, A.~F.}
\newblock Uniqueness and stability of entropy shocks to the isentropic {E}uler system in a class of inviscid limits from a large family of {N}avier-{S}tokes systems.
\newblock {\em Invent. Math. 224}, 1 (2021), 55--146.

\bibitem{KVW}
{\sc Kang, M.-J., Vasseur, A.~F., and Wang, Y.}
\newblock {$L^2$}-contraction of large planar shock waves for multi-dimensional scalar viscous conservation laws.
\newblock {\em J. Differential Equations 267}, 5 (2019), 2737--2791.

\bibitem{K1}
{\sc Kru{\v{z}}kov, S.~N.}
\newblock First order quasilinear equations with several independent variables.
\newblock {\em Mat. Sb. (N.S.) 81 (123)\/} (1970), 228--255.

\bibitem{LiuBook}
{\sc Liu, T.-P.}
\newblock {\em Shock Waves}.
\newblock Graduate Studies in Mathematics. Volume: 215. American Mathematical Society, 2021.

\bibitem{Vasseur16}
{\sc Vasseur, A.~F.}
\newblock Relative entropy and contraction for extremal shocks of conservation laws up to a shift.
\newblock In {\em Recent advances in partial differential equations and applications}, vol.~666 of {\em Contemp. Math.} Amer. Math. Soc., Providence, RI, 2016, pp.~385--404.

\bibitem{Wang22}
{\sc Wang, T., and Wang, Y.}
\newblock {Nonlinear stability of planar viscous shock wave to three-dimensional compressible Navier-Stokes equations}.
\newblock {\em arXiv:2204.09428, To appear in JEMS\/} (2024).

\end{thebibliography}
\end{document}